\documentclass[tikz,12pt]{article}

\usepackage[utf8]{inputenc}
\usepackage[english]{babel}
\usepackage{hyperref}
\usepackage{amsmath,amsthm,enumerate}
\usepackage{enumitem} 
\usepackage{amssymb}
\usepackage{graphicx, color}
\usepackage{epsfig}
\usepackage{tikz}
\usetikzlibrary{calc,decorations.pathmorphing,decorations.text}
\usepackage{subcaption}
\usepackage{refcount}
\usepackage[hmargin=2.5cm,vmargin=3cm]{geometry}
\usepackage{mathtools, braket}
\usepackage{authblk}
\usepackage{centernot}

\theoremstyle{plain}
\newtheorem{thm}{Thm}[section]

\newtheorem{theorem}[thm]{Theorem}
\newtheorem{lemma}[thm]{Lemma}
\newtheorem{corollary}[thm]{Corollary}
\newtheorem{proposition}[thm]{Proposition}

\newtheorem{problem}[thm]{Problem}

\newtheorem{observation}[thm]{Observation}

\usepackage{soul,color} 
\definecolor{lightgreen}{RGB}{153,255,153}

\definecolor{lightyellow}{RGB}{255,252,187}

\definecolor{brightred}{RGB}{255,80,102}
 
\definecolor{lightpurple}{RGB}{255,204,255}

\soulregister\cite7
\soulregister\ref7
\soulregister\pageref7

\usepackage[colorinlistoftodos,prependcaption,textsize=scriptsize,color=olive!70, textwidth=30mm,]{todonotes}

\newenvironment{proof*}{\noindent\emph{Proof of the claim:}}{\hfill$\Diamond$}

\newcommand{\spto}{\ensuremath{\stackrel{s.p.}{\longrightarrow}}}

\newcommand{\mad}{{\rm mad}}

\usepackage{caption}
\usetikzlibrary{matrix}
\usetikzlibrary{decorations.markings}
\tikzstyle{vertex}=[circle, draw, fill=black!50,
inner sep=0pt, minimum width=4pt]
\tikzset{->-/.style={decoration={
			markings,
			mark=at position .5 with {\arrow{>}}},postaction={decorate}}}

\tikzstyle{bigblue}=[color=blue, very thick, >=stealth]
\tikzstyle{lightblue}=[color=blue, thin, >=stealth]

\tikzstyle{bigred}=[color=red, very thick, >=stealth]
\tikzstyle{lightred}=[color=red, thin, >=stealth]

\tikzstyle{biggreen}=[color=black!30!green, very thick, >=stealth]
\tikzstyle{lightgreen}=[color=black!30!green,  thin, >=stealth]

\begin{document}
	
	\title{Mapping sparse signed graphs to  $(K_{2k}, M)$}

	\author[1]{Reza Naserasr} 
	\author[2]{Riste \v{S}krekovski}
	\author[1]{Zhouningxin Wang}
	\author[1,3]{Rongxing Xu} 
	
	\affil[1]{Universit\'{e} de Paris, CNRS, IRIF, F-75006, Paris, France. E-mail addresses: \{reza, wangzhou4\}@irif.fr.}
	\affil[2]{Faculty of Mathematics and Physics, University of Ljubljana, Ljubljana, Slovenia. Email address: skrekovski@gmail.com.}
	\affil[3]{Department of Mathematics, Zhejiang Normal University, Jinhua, China. Email address: xurongxing@yeah.net.}
	
	
	
	\maketitle
	
\begin{abstract}
A homomorphism of a signed graph $(G, \sigma)$ to $(H, \pi)$ is a mapping of vertices and edges of $G$ to (respectively) vertices and edges of $H$ such that adjacencies, incidences and the product of signs of closed walks are preserved.
Motivated by reformulations of the $k$-coloring problem in this language, and specially in connection with results on $3$-coloring of planar graphs, such as  Gr\"otzsch's theorem, in this work we consider bounds on maximum average degree which are sufficient for mapping to the signed graph $(K_{2k}, \sigma_m)$ ($k\geq 3$) where $\sigma_m$ assigns to edges of a perfect matching the negative sign. For $k=3$, we show that the maximum average degree strictly less than $\frac{14}{5}$ is sufficient and that this bound is tight. For all values of $k\geq 4$, we find the best maximum average degree bound to be 3. 

While the homomorphisms of signed graphs is relatively new subject, through the connection with the homomorphisms of $2$-edge-colored graphs, which are largely studied, some earlier bounds are already given. In particular, it is implied from Theorem 2.5 of ``Borodin, O. V., Kim, S.-J., Kostochka, A. V., and West, D. B., Homomorphisms from sparse graphs with large girth. J. Combin. Theory Ser. B (2004)" that if $G$ is a graph of girth at least 7 and maximum average degree $\frac{28}{11}$, then for any signature $\sigma$ the signed graph $(G,\sigma)$ maps to $(K_6, \sigma_m)$. 

We discuss applications of our work to signed planar graphs and, among others, we propose questions similar to Steinberg's conjecture for the class of signed bipartite planar graphs.

	\end{abstract}
	
	\section{Introduction}\label{Int}

	A \emph{signed graph} $(G, \sigma)$ is a graph $G$ together with an assignment $\sigma$ which assigns to each edge one of the two signs: positive $(+)$ or negative $(-)$. Given a signed graph on $G$ where $E^-$ is the set of negative edges, we may equivalently use $(G, E^-)$ to denote this signed graph. When presenting signed graph in figures, we use solid or blue lines to represent positive edges and dashed or red lines to represent negative edges. For (underlying) graphs, we use gray color. A \emph{subgraph} of signed graph $(G, \sigma)$ is a subgraph $H$ of $G$ together with the signature induced by $\sigma$ on $H$. For simplicity, such a subgraph will be denoted by $(H, \sigma)$. 
	
	Given a signed graph $(G, \sigma)$ and an edge cut $(X, V(G)\setminus X)$, a \emph{switching} at $X$ is to multiply the sign of each edge in the cut by a $-$. 
	When $X=\{v\}$, we may simply say a switching at $v$. A switching at $X$ is the result as switching at all vertices of $X$.
    Two signed graphs $(G, \sigma)$ and $(G, \sigma')$ are said to be \emph{switching equivalent} if one is obtained from the other by a switching.
	It is easily observed that this is an equivalence relation on the set of all signatures on a given graph $G$. 
	A cycle or a closed walk of $G$ is said to be \emph{positive} in $(G, \sigma)$ if the product of signs of all its edges (considering multiplicity) is positive, it is said to be \emph{negative} otherwise. Observe that the sign of a cycle or a closed walk after a switching (on a set of vertices) remains the same. The converse is also true in the following sense.
	
	\begin{theorem}{\em \cite{Za82}}\label{thm:Zaslavsky}
		Signed graphs $(G, \sigma)$ and $(G, \sigma')$ are switching equivalent if and only if they have a same set of positive cycles.
	\end{theorem}
	
	Considering the parity of the length of a closed walk and the sign of it, there are four possible types of closed walks: type $00$ is a closed walk which is positive and of even length, type $01$ is positive and odd, type $10$ is negative and even, finally type $11$ is negative and odd. The length of a shortest nontrivial closed walk in $(G, \sigma)$ of type $ij$, ($ij\in \mathbb{Z}_2^2$), is denoted by $g_{ij}(G, \sigma)$. We write $g_{ij}(G, \sigma) =\infty$ when there is no closed walk of type $ij$ in $(G,\sigma)$.
	
	A \emph{homomorphism} of a signed graph $(G, \sigma)$ to a signed graph $(H, \pi)$, also referred to as \emph{switch homomorphism}, is a mapping of vertices and edges of $G$ to the vertices and edges of $H$ such that adjacencies, incidences and signs of closed walks are preserved. When there exists such a homomorphism, we write $(G, \sigma) \to (H, \pi)$.
	An immediate corollary of the definition is the following no-homomorphism lemma.
	
	\begin{lemma}\label{lem:Nohom}
		If $(G, \sigma) \to (H, \pi)$, then $g_{ij}(G, \sigma) \geq g_{ij}(H, \pi)$ for $ij\in \mathbb{Z}_2^2$.
	\end{lemma}
	
	Thus $g_{ij}(G, \sigma) \geq g_{ij}(H, \pi)$ is a necessary condition for the existence of a homomorphism of $(G, \sigma)$ to $(H, \pi)$. The following question then is one of the most central questions in graph theory.
	
	\begin{problem}\label{prob:NecessarySufficient?}
		Given a signed graph $(H, \pi)$, under which structural condition on $G$ the necessary conditions of Lemma~\ref{lem:Nohom} becomes sufficient?
	\end{problem}
	
	For example, the Four-Color Theorem states that for $(K_4, \emptyset)$ with structural condition of ``planarity", the necessary conditions of the no-homomorphism lemma are also sufficient. The conditions of no-homomorphism lemma hold as long as $(G, \sigma)$ has no loop and no negative cycle, in which case any proper $4$-coloring of $G$ (which for the planar graphs is provided by the Four-Color Theorem) can be regarded as a homomorphism of $(G, \sigma)$ to $(K_4, \emptyset)$. We will discuss a strengthening of this particular case which is the main motivation of our work.
	On the other hand, for $(K_3, \emptyset)$, even with the extra condition of planarity, not only the necessary conditions of the no-homomorphism lemma are not sufficient but it is expected to be far from it, as $3$-coloring of planar graphs is known to be an NP-hard problem in \cite{GJS76}. For such cases, two closely related conditions are considered, the first is having high girth for planar graphs and the second is to have low \emph{maximum average degree}. Recall that maximum average degree of a graph $G$, denoted by $\mad(G)$, is the maximum average degrees taken over all subgraphs of $G$. This approach is considered in \cite{CNS20} where it is shown that:
	
	\begin{theorem}
		Given a signed graph $(H, \pi)$, there exists an $\epsilon>0$ such that every signed graph $(G, \sigma)$, satisfying $g_{ij}(G, \sigma) \geq g_{ij}(H, \pi)$ and $\mad(G)< 2+\epsilon$, admits a homomorphism to $(H, \pi)$.
	\end{theorem}
	
	A main question then is to find the best value of $\epsilon$ for a given signed graph $(H, \pi)$. In \cite{CNS20}, the best value of $\epsilon=\frac{4}{7}$ is proved for $(K_4, e)$ where only one edge is negative. In this work, then we find the best value for two classes of signed graphs: $(K_{2n}, M)$ and $(K_{n, n}, M)$ where $M$ is a perfect matching of the graph under consideration. 
	
	In Section~\ref{sec:HomomorphismToKkk}, we will explain our motivation of choosing these two families of signed graphs, noting that $n=3$ is of special importance and is the main case of the difficulty in this work.
	
	\subsection{Edge-sign preserving homomorphism and Double Switching Graph}\label{sec:DSG}
	
	A homomorphism $\phi$ of $(G, \sigma)$ to $(H, \pi)$ is said to be \emph{edge-sign preserving} if edge $\phi(u)\phi(v)$ of $(H, \pi)$ has the same sign as $uv$ for each edge $uv$ of $(G, \sigma)$. We write $(G, \sigma) \spto (H, \pi)$ to denote the existence of an edge-sign preserving homomorphism of $(G, \sigma)$ to $(H, \pi)$. When we consider the edge-sign preserving homomorphism, we also say a path is \emph{positive} if the product of signs of all its edges is positive, it is \emph{negative} otherwise.

	\begin{theorem}\emph{\cite{NSZ21}}\label{thm:Sign-preserving}
		For signed graphs $(G, \sigma)$ and $(H,\pi)$, $(G, \sigma)\to (H, \pi)$ if and only if there exists a switching equivalent $(G, \sigma')$ of $(G, \sigma)$ such that $(G, \sigma')\spto (H, \pi)$.   
	\end{theorem}
	
	A strong relation between homomorphism of signed graphs and edge-sign preserving homomorphism of signed graphs is provided based on the following notion.
	
	Given a signed graph $(G,\sigma)$, the \emph{Double Switching Graph}, denoted ${\rm DSG}(G,\sigma)$, is a signed graph built as follows. If $V=\{x_1, x_2, \ldots, x_n\}$ is the vertex set of $G$, then we have two disjoint copies of it, $V^{+}=\{x_1^{+}, x_2^{+}, \ldots, x_n^{+}\}$ and $V^{-}=\{x_1^{-}, x_2^{-}, \ldots, x_n^{-}\}$ in ${\rm DSG}(G,\sigma)$.
	Each set of vertices then induces a copy of $(G, \sigma)$, furthermore, a vertex $x_i^{-}$ connects to vertices in $V^{+}$ as it is obtained from a switching on $x_i$, more precisely, if $x_ix_j$ is a positive (negative) edge in $(G, \sigma)$, then $x_i^{+}x_j^{+}$ and $x_i^{-}x_j^{-}$ are positive (negative) edges in ${\rm DSG}(G, \sigma)$, and $x_i^{+}x_j^{-}$ and $x_i^{-}x_j^{+}$ are negative (positive) edges in ${\rm DSG}(G, \sigma)$.
	
	The connection, shown in Theorem~\ref{thm:Hom->Edge-SignPres}, was originally proved in \cite{BG09}. This connection comes through an intermediary definition of switching homomorphism. The terminology is from \cite{NSZ21}. For the sake of completeness, we give a direct proof from the definition we work with.
	
	\begin{theorem}\label{thm:Hom->Edge-SignPres}
		Given signed graphs $(G, \sigma)$ and $(H, \pi)$, we have $(G, \sigma) \to (H, \pi)$ if and only if  $(G, \sigma) \spto {\rm DSG}(H,\pi)$.
	\end{theorem}
	
	\begin{proof}
		An identity mapping of the subgraph of ${\rm DSG}(H,\pi)$ induced on $V^{+}$ can be extended to a mapping of ${\rm DSG}(H,\pi)$ to $(H,\pi)$ as follows. Each vertex $x_i^{-}$ is mapped to $x_i^{+}$, each edge $x_i^{+}x_j^{-}$ is mapped to an edge $x_i^{+}x_j^{+}$ of opposite sign and each edge $x_i^{-}x_j^{-}$ is mapped to an $x_i^{+}x_j^{+}$ of a same sign. It is easily verified that this mapping is a homomorphism of ${\rm DSG}(H,\pi)$ to $(H,\pi)$. Thus any edge-sign preserving homomorphism of $(G,\sigma)$ to ${\rm DSG}(H,\pi)$ induces a homomorphism of $(G, \sigma)$ to $(H, \pi)$.
		
		For the inverse, assume $(G, \sigma)$ maps to $(H, \pi)$ and let $\phi$ be such a mapping. Let $\sigma'$ be a signature on $G$ where the sign of each edge $uv$ is the same as the sign of $\phi(u)\phi(v)$.  Then, clearly, the image of each cycle $C$ of $G$ in $(H, \phi)$ is a closed walk whose sign (with respect to $\pi$) is the same as the sign of $C$ in $(G, \sigma')$. As this is also the case for $(G, \sigma)$ (by our definition of homomorphism), it follows from Theorem~\ref{thm:Zaslavsky} that $\sigma'$ is switching equivalent to $\sigma$. Thus there is a set $X$ of vertices such that $\sigma'$ is obtained from $\sigma$ by a switching at $X$. We modify $\phi$ to a mapping $\psi$ of $(G, \sigma)$ to ${\rm DSG}(H,\pi)$ as follows. If $v\notin X$, and $\phi(v)=x_i$, then $\psi(v)=x_i^{+}$. If $v\in X$, and $\phi(v)=x_i$, then $\psi(v)=x_i^{-}$. One may now easily verify that $\psi$ is an edge-sign preserving homomorphism of $(G, \sigma)$ to ${\rm DSG}(H,\pi)$.	
	\end{proof}

	\section{Homomorphism to $(K_{k,k}, M)$ and $(K_{2k}, M)$}\label{sec:HomomorphismToKkk}
	
	\subsection{Chromatic number and homomorphism to $(K_{k,k}, M)$}\label{sec:Kkk}
	
	The following construction is introduced in \cite{NRS15} to connect the notion of chromatic number of graphs to the homomorphism of signed bipartite graphs.
	
	Given a graph $G$, a signed graph $S(G)$ is built as follows: starting from the vertex set of $G$, for each edge $uv$ we add two more vertices $x_{uv}$ and $y_{uv}$ and join them to $u$ and $v$ (noting that the original edge $uv$ is not an edge of $S(G)$), finally for each $4$-cycle $ux_{uv}vy_{uv}$ we assign a negative sign to one of the edges. It is then proved in \cite{NRS15} that:
	
	\begin{theorem}\label{thm:X->S(G)}
		Given a graph $G$, it is bipartite if and only if $S(G)$ maps to $(K_{2,2}, e)$ (one negative edge) and it is $k$-colorable, $k\geq 3$, if and only if $S(G)\to (K_{k,k},M)$.  	
	\end{theorem}

	Thus the signed graphs $(K_{k,k},M)$ are of special importance in the study of homomorphism of signed graphs. 
	
	Using Theorem~\ref{thm:X->S(G)}, a restatement of the Four-Color Theorem is as follow: 
	
	\begin{theorem}{\rm [The Four-Color Theorem restated]} 
		For every planar (simple) graph $G$, $S(G)\to (K_{4,4},M)$. 
	\end{theorem}
	
	Observing that when $G$ is planar, $S(G)$ is a signed bipartite planar graph (where in one part all vertices are of degree 2). The following strengthening of the Four-Color Theorem is given in \cite{NRS15} (proof of which is based on an edge-coloring result of B. Guenin which in turn is based on the Four-Color Theorem).

	\begin{theorem}\label{thm:4CTStrengthened}
		Any signed planar graph $(G, \sigma)$ satisfying that $g_{ij}(G, \sigma)\geq g_{ij} (K_{4,4}, M)$ for $ij\in \mathbb{Z}^2_2$ maps to $(K_{4,4},M)$. 
	\end{theorem}
	
	The conditions of no-homomorphism lemma for $(K_{4,4,} M)$, i.e., that $g_{ij}(G, \sigma)\geq g_{ij} (K_{4,4}, M)$, imply that $G$ is a simple bipartite graph. Thus this statement is a reformulation of the claim that every signed bipartite planar (simple) graph maps to $(K_{4,4,} M)$. 
	
	Viewing the Four-Color Theorem as the mapping of planar graphs to $K_4$, one may consider the only three core (proper) subgraphs $K_1$, $K_2$, $K_3$ of $K_4$ and for each of them, asks: which planar graphs map to these subgraphs?
	While the homomorphism problems to $K_1$ is a matter of triviality and to $K_2$ is rather easy, the homomorphism problem from planar graphs to $K_3$ has been a subject of great study. On the one hand, it is proved to be an NP-complete problem. On the other hard, starting with Gr\"otzsch's theorem, an extensive family of planar graphs are proved to be $3$-colorable. We refer to \cite{DKT11}, \cite{DKT20}, \cite{D13}, \cite{KY14}, \cite{T03} and to the references there for some of the work on this subject.
	
	It is then natural to ask for each core subgraphs of $(K_{4,4},M)$ which families of planar graphs map to. Of such subgraphs of $(K_{4,4},M)$ there are two notable one to consider: 1. the negative $4$-cycle. We refer to \cite{NPW20} for recent progress on this problem. 2. $(K_{3,3}, M)$.
 Considering Theorem~\ref{thm:X->S(G)}, the question of mapping signed bipartite planar graphs captures $3$-coloring problem of planar graphs. For example, Gr\"otzsch's theorem could be restated as: for any triangle-free planar graph $G$, $S(G)\to (K_{3,3,}, M)$. This work then is motivated by the study of mapping signed bipartite planar graphs to $(K_{3,3}, M)$. Noting that this is an NP-complete problem, we search for (polynomial-time verifiable) conditions that are sufficient for such mappings. Among other results, we prove that if $G$ has maximum average degree less than $\frac{14}{5}$, and it satisfies the conditions of the no-homomorphism lemma with respect to $(K_{3,3},M)$, then $(G, \sigma)$ maps to $(K_{3,3}, M)$. Application to planar graphs then is considered in the last section.

	\subsection{Chromatic number and homomorphism to $(K_{2k}, M)$} \label{sec:K2kIntro}
	
	Motivated by the following theorem, we will change our target graph to be $(K_6, M)$ and more generally $(K_{2k}, M)$. 
	
	\begin{theorem}\label{thm:K_{k,k}toK_{2k}}
		A signed bipartite graph $(G, \sigma)$ maps to $(K_{k,k}, M)$ if and only if it maps to $(K_{2k}, M)$.
	\end{theorem}
	
	\begin{proof}
	We will view $(K_{k,k}, M)$ as a subgraph of $(K_{2k}, M)$. If $(G, \sigma)$ maps to $(K_{k,k}, M)$, then it, obviously, maps to $(K_{2k}, M)$ as well. 
For the other direction, assume that $\phi$ is a mapping of $(G, \sigma)$ to $(K_{2k}, M)$, and let $(X,Y)$ be a bipartition of $G$ and $(A,B)$ be a bipartition of $K_{k,k}$. If $\phi$ maps each vertex of $X$ to a vertex in $A$ and each vertex of $Y$ to a vertex in $B$, then we are done. Otherwise, we define a mapping $\phi'$ as follows: 
		\[
		\phi'(v) = \begin{cases} \phi(v), &\text{ if $v\in X, \phi(v)\in A$ or $v\in Y,\phi(v)\in B$ } \cr
		m(\phi(v)), &\text{ otherwise}. \cr
		\end{cases}
		\] 
		where $m(u)$ is the match of $u$ by $M$. Since $G$ is bipartite, $\phi'(v)$ preserves incidences and adjacencies. The images of closed walks then are the same as in $\phi$. Thus $\phi'$ is a homomorphism of $(G,\sigma)$ to $(K_{k,k}, M)$.
	\end{proof}
	
	Combining Theorems~\ref{thm:K_{k,k}toK_{2k}} and \ref{thm:X->S(G)} we have the following corollary.
	
	\begin{corollary}\label{coro:X->K_{2k}}
		Given an integer $k\geq 3$ and a graph $G$, we have $\chi(G)\leq k$ if and only if $S(G)\to (K_{2k}, M)$.
	\end{corollary}
	
	However, in contrast to Theorem~\ref{thm:4CTStrengthened}, we will show that the conditions of no-homomorphism lemma are not sufficient for mapping signed planar graphs to $(K_8, M)$. In Section~\ref{sec:examples}, we will give a series of signed planar graphs of girth $3$ that do not map to $(K_8, M)$. (see Theorem~\ref{thm:K8tight})

	\subsection{Mapping to $(K_6, M)$}\label{sec:main}
	
	A main focus of this work is the following theorem. We note that since we only consider simple signed graphs, the conditions of no-homomorphism lemma with respect to $(K_{6}, M)$ is always satisfied.
	
	\begin{theorem}\label{main-theorem}
		Every signed graph with maximum average degree less than $\frac{14}{5}$ admits a homomorphism to $(K_6, M)$. Moreover, the bound $\frac{14}{5}$ is the best possible.
	\end{theorem}
	
	In light of Theorem~\ref{thm:Hom->Edge-SignPres}, we will study edge-sign preserving homomorphism to ${\rm DSG}(K_6, M)$ in place of the homomorphism to $(K_6, M)$. Our proof is based on discharging technique. To provide a set of forbidden configuration, we first develop some ${\rm DSG}(K_6, M)$-list-coloring tools in Section~\ref{sec:precoloring}. Assuming $(G, \sigma)$ is a minimum counterexample to Theorem~\ref{main-theorem}, our set of forbidden configurations are provided in Section~\ref{sec:forbidden-configurations}. Then in Section~\ref{sec:discharging}, the discharging technique is employed to prove that $(G, \sigma)$ cannot have maximum average degree less than $\frac{14}{5}$. Generalization of the theorem to $(K_{2k},M)$ is considered in Section~\ref{sec:K2k} and examples toward the tightness of results are provided in Section~\ref{sec:examples}.

    We shall note that, while the notion of homomorphism of signed graphs is relatively new, the notion of edge-sign preserving homomorphism is a renaming of the notion of homomorphism of $2$-edge-colored graphs which are largely studies since 1980's. In particular, in relation to our work, we may apply Theorem 2.5 of \cite{BKKW04} to ${\rm DSG}(K_6,M)$ with $t=3$ to obtain the following:

\begin{theorem}[Special case of Theorem 2.5 of \cite{BKKW04}]
If $G$ is a graph of girth at least $7$ and maximum average degree at most $\frac{28}{11}$, then $(G, \sigma)\to (K_6,M)$ for any signature $\sigma$.
\end{theorem}

	\section{Preliminary}\label{sec:Pre}
		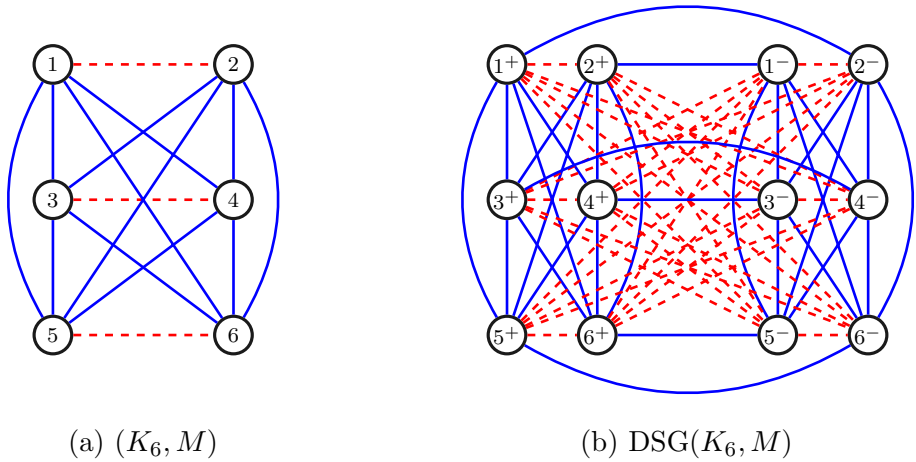
\begin{figure}[htpb]
		\begin{minipage}[t]{.5\textwidth}
			\centering
			\begin{tikzpicture}[>=latex,
			roundnode/.style={circle, draw=black!90, very thick, minimum size=5mm, inner sep=0pt},
			scale=1.2]
			
			\node[roundnode](a1) at (1,3){\scriptsize $1$};
			\node[roundnode](b1) at (3,3) {\scriptsize $2$};
			\node[roundnode](c1) at (1,1.5) {\scriptsize $3$};
			\node[roundnode](d1) at (3,1.5) {\scriptsize $4$};
			\node[roundnode](e1) at (1,0){\scriptsize $5$};
			\node[roundnode](f1) at (3,0){\scriptsize $6$};
			\draw[red, dashed, line width =1pt] (a1)--(b1);
			\draw[red, dashed, line width =1pt](c1)--(d1);
			\draw[red, dashed, line width =1pt] (e1)--(f1);
			\draw[blue, line width =1pt](a1)--(c1)--(e1)--(b1)--(d1)--(f1)--(a1)--(d1)--(e1);
			\draw[blue, line width =1pt](f1)--(c1)--(b1);
			\draw[blue, line width =1pt](a1) edge[bend right] (e1);
			\draw[blue, line width =1pt](b1) edge[bend left] (f1);
			
			\node[roundnode, white](e3) at (0,0){\scriptsize $5^+$};
			\node[roundnode, white](f2) at (4,0){\scriptsize $6^-$};
			\draw[line width =1pt, white](e3) edge[bend right] (f2);
			
			\end{tikzpicture}
			\subcaption{$(K_6,M)$}
			\label{fig:K6}
		\end{minipage} 
		\begin{minipage}[t]{.4\textwidth}
			\centering
			\begin{tikzpicture}[>=latex,
			roundnode/.style={circle, draw=black!90, very thick, minimum size=5mm, inner sep=0pt},  scale=1.2
			]
			\node[roundnode](a1) at (1,3){\scriptsize $1^+$};
			\node[roundnode](b1) at (2,3) {\scriptsize $2^+$};
			\node[roundnode](c1) at (1,1.5) {\scriptsize $3^+$};
			\node[roundnode](d1) at (2,1.5) {\scriptsize $4^+$};
			\node[roundnode](e1) at (1,0){\scriptsize $5^+$};
			\node[roundnode](f1) at (2,0){\scriptsize $6^+$};
			\draw[red, dashed, line width =1pt] (a1)--(b1);
			\draw[red, dashed, line width =1pt](c1)--(d1);
			\draw[red, dashed, line width =1pt] (e1)--(f1);
			\draw[blue, line width =1pt](a1)--(c1)--(e1)--(b1)--(d1)--(f1)--(a1)--(d1)--(e1);
			\draw[blue, line width =1pt](f1)--(c1)--(b1);
			\draw[blue, line width =1pt](a1) edge[bend right] (e1);
			\draw[blue, line width =1pt](b1) edge[bend left] (f1);
			
			\node[roundnode](a2) at (4,3){\scriptsize $1^-$};
			\node[roundnode](b2) at (5,3) {\scriptsize $2^-$};
			\node[roundnode](c2) at (4,1.5) {\scriptsize $3^-$};
			\node[roundnode](d2) at (5,1.5) {\scriptsize $4^-$};
			\node[roundnode](e2) at (4,0){\scriptsize $5^-$};
			\node[roundnode](f2) at (5,0){\scriptsize $6^-$};
			
			\draw[red, dashed, line width =1pt] (a2)--(b2);
			\draw[red, dashed, line width =1pt](c2)--(d2);
			\draw[red, dashed, line width =1pt] (e2)--(f2);
			\draw[blue, line width =1pt](a2)--(c2)--(e2)--(b2)--(d2)--(f2)--(a2)--(d2)--(e2);
			\draw[blue, line width =1pt](f2)--(c2)--(b2);
			\draw[blue, line width =1pt](a2) edge[bend right] (e2);
			\draw[blue, line width =1pt](b2) edge[bend left] (f2);
			
			\foreach \i in {c,d,e,f}{
				\draw[red, dashed, line width =1pt](a1)--(\i2);
				\draw[red, dashed, line width =1pt](b1)--(\i2);}
			
			\foreach \i in {a,b,e,f}{
				\draw[red, dashed, line width = 1pt](c1)--(\i2);
				\draw[red, dashed, line width = 1pt](d1)--(\i2);}
			
			\foreach \i in {a,b,c,d}{
				\draw[red, dashed, line width = 1pt](e1)--(\i2);
				\draw[red, dashed, line width = 1pt](f1)--(\i2);}
			
			\draw[blue, line width=1pt](a1) edge[bend left] (b2);
			\draw[blue, line width=1pt](a2)--(b1);
			
			\draw[blue, line width=1pt](c1) edge[bend left] (d2);
			\draw[blue, line width=1pt](c2)--(d1);
			
			\draw[blue, line width=1pt](e1) edge[bend right] (f2);
			\draw[blue, line width=1pt](e2)--(f1);
			\end{tikzpicture}
			\subcaption{${\rm DSG}(K_6,M)$}
		\end{minipage} 
		\caption{Signed graphs $(K_{6},M)$ and ${\rm DSG}(K_6,M)$}
		\label{fig:DSG(K6M)}
	\end{figure}

    Vertices of $(K_{6}, M)$ and its Double Switching Graph ${\rm DSG}(K_{6}, M)$ are labeled as in Figure~\ref{fig:DSG(K6M)}.  The signature of ${\rm DSG}(K_{6}, M)$ will be denoted by $m^{*}$. Thus in $(K_{6}, M)$, vertices $2i-1$ and $2i$ are connected by a negative edge ($1\leq i \leq 3$) where all other edges are positive.  The vertex set of ${\rm DSG}(K_{6}, M)$, which will be denoted by $C$, is partitioned into two sets: $C^+$ and $C^-$. Each vertex $x$ in $C^{\alpha}$, $\alpha \in \{+, -\}$, is connected to a unique vertex in $C^{\alpha}$ by a negative edge, this vertex is called the \emph{pair} of $x$, denoted by \emph{Pair$(x)$}. In the rest of this work, a pair of colors would refer to a vertex of ${\rm DSG}(K_6,M)$ and its pair. Given a vertex $x^{\alpha}$ and its pair $y^{\alpha}$, the vertices $x^{-\alpha}$ and $y^{-\alpha}$ form another pair and the four vertices together form a \emph{layer} in ${\rm DSG}(K_{6}, M)$ (a horizontal line in Figure~\ref{fig:DSG(K6M)}). Given a vertex $x\in C$, the only vertex not adjacent to it is called \emph{inverse} of $x$. Moreover, given a subset $X$ of $C$, the inverse of $X$, denoted $X^{-}$, is the set of of inverses of the elements of $X$.
	
	One of the main ideas of this work is to extend a partial mapping of a signed graph $(G,\sigma)$ to ${\rm DSG}(K_{6}, M)$ to a mapping of the full graph. This is captured in our figures with the following setting: vertices presented in square are precolored and circular vertices are yet to be colored. Such extension problems lead to the idea of $(K_6,M)$-list coloring or ${\rm DSG}(K_6,M)$-list coloring. We refer to \cite{BBFHJ20} and references there for a general study of list-homomorphism problem of graphs and signed graphs. List assignments considered in this work, however, are rather restricted. To better describe lists of available colors on vertices, we introduce the following terminology.
	
	A set $L$ of colors ($L\subseteq C$) is said to be \emph{paired} if for all but at most one $x\in L$ we have $Pair(x)\in L$. For example, $L_1=\{1^+, 2^+, 3^+, 4^+, 5^+\}$ and $L_2=\{1^+, 2^+, 3^+, 5^-, 6^-\}$ are paired sets of colors while  $L_3=\{1^+, 3^+\}$ is not. We say a paired set $X$ is \emph{layered} if no three colors of $X$ belong to a layer. We say a layered set $X$ is \emph{one-sided} if all the colors in $X$ are on the same side, i.e., either $X\subseteq  C^{+}$ or $X\subseteq C^{-}$. A layered set $X$ of size $2k+1$, $k=0,1,2$, is said to be a \emph{neighbored $(2k+1)$-set} if $C$ consists of $k$ pairs on one side and a single element on the other side. Observe that a neighbored $5$-set is the set of all vertices adjacent to a vertex $v$ by positive edges. A neighbored $3$-set consists of neighbors of a vertex $x$ which are connected to $x$ by positive edges and each of which has a positive path of length 2 to another fixed vertex $y$.

	Given a signed graph $(G,\sigma)$, a \emph{${\rm DSG}(K_6,M)$-list assignment} $L$ of $(G,\sigma)$ is a function that assigns to each vertex of $G$ a set $L(v) \subseteq C$. An \emph{edge-sign preserving $L$-homomorphism} of $(G,\sigma)$ to ${\rm DSG}(K_6,M)$ is a mapping $\phi:V(G) \to C$ such that for each vertex $v \in V(G)$, $\phi(v) \in L(v)$ and for each edge $uv \in E(G)$, $m^*(\phi(u)\phi(v)) = \sigma(uv)$.

	In this paper, we consider only lists that are subsets of $C$. When there is no confusion, we call such list homomorphism \emph{$L$-coloring}. If there exists an $L$-coloring, we say $(G,\sigma)$ is \emph{$L$-colorable}. 
	
	For a given signed graph $(G, \sigma)$, a signature on $G$ obtained from $\sigma$ by switching at a vertex set $X$, is denoted by $\sigma^X$. Given a ${\rm DSG}(K_6, M)$-list assignment $L$ of $(G, \sigma)$, let $L^X$ be a list assignment defined by:
	\[
	L^X(v)= \begin{cases} (L(v))^-, &\text{ for $v\in X$} \cr
	L(v), &\text{ for $v\in V(G)\setminus X$}. \cr
	\end{cases}
	\] 
	
\begin{observation}\label{obs:L-Lx}
A signed graph $(G, \sigma)$ is $L$-colorable if and only if $(G, \sigma^X)$ is $L^X$-colorable.  
\end{observation}

\subsection{$L$-coloring of a signed rooted tree}	

Let $T$ be a rooted tree with root $v$. Given a vertex $x$ of $T$, we define a subtree rooted at $x$, denoted $T_x$, to be the subgraph induced by $x$ and those vertices of $T$ whose unique path to $v$ contains $x$. Let $L$ be a ${\rm DSG}(K_6, M)$-list assignment of $(T, \sigma)$. 

Toward deciding if $T$ is $L$-colorable, and taking advantage of the rooted tree, we have the following definitions:

For a vertex $x$ of $T$, we define the set of \emph{admissible colors}, denoted $L^a(x)$, to be the set of the colors $c\in L(x)$ such that with the restriction of $L$ onto $T_x$ there exists an $L$-coloring $\phi$ of $T_x$ where $\phi(x)=c$. Thus $T$ is $L$-colorable if and only if $L^a(v)\neq \emptyset$ and moreover, reducing $L(x)$ to $L^a(x)$ at any time would not affect $L$-colorability of $T$.

Sometimes, instead of considering the set of admissible colors at $x$, it is preferable to consider the set of colors that are forbidden through the children or neighbors in general. Let $xy$ be an edge of $(G, \sigma)$ and assume $L(y)$ is the set of colors available at $y$. Then we define $F_{L(y)}(x)$ to be the set of forbidden colors at $x$ because of the edge $xy$ and the list $L(y)$. More precisely, a color $c$ is forbidden on $x$ because of $L(y)$ if for each choice $c'\in L(y)$ either $c$ is not adjacent to $c'$ or $\sigma(xy)\neq m^*(cc')$. For example if $L(y)=\{1^+,2^+, 3^-\}$, and $xy$ is a positive edge, then $F_{L(y)}(x)=\{1^+,2^+,3^-,4^-\}$.  When the list assignment is clear from the context, we may simply write $F_y(x)$ in place of $F_{L(y)}(x)$. For $L$-coloring of a subtree $T$, we have the following relation between the two notions: $$L^a(x)=L(x) \setminus \bigcup_{\substack{y \\ y\text{ child of } x}}F_{L^a(y)}(x). $$

Thus in the rest of this work, at any time we may modify list $L(x)$ at anytime by removed colors from $F_y(x)$ or simply replace it by $L^a(x)$. In the following lemmas, we gather basic rules on theses modifications.

\begin{lemma}\label{lem:2-Path}
	Let $T$ be a path $vv_1v_2$ where $v$ is the root.  Let $L$ be a list assignment satisfying $L(v)=L(v_1)=C$, and $|L(v_2)|=1$ (i.e., $v_2$ is precolored). Then $L^a(v)$ is a paired $10$-set.
\end{lemma} 

\begin{lemma}
	Let $T$ be a rooted tree on $3$-path $v_1v_2vv_3$ with $v$ being the root. Let $L$ be a list assignment satisfying $|L(v_1)|=|L(v_3)|=1$ and $L(v_2)=L(v)=C$. Then $L^a(v)$ is either a neighbored $5$-set, or a neighbored $3$-set, or it is a one-sided paired $4$-set.
\end{lemma}

\begin{lemma}\label{lem:3-1}
	Let $(T, \sigma)$ be a signed rooted tree of Figure~\ref{fig:3-1} with $v$ as the root. Let $L$ be a list assignment satisfying $|L(v_1')|=|L(v_2)|=1$ and $L(v)=L(v_0)=L(v_1)=C$.  
	Then we have $|L^a(v)|\geq 8$.
\end{lemma}

\begin{figure}[h]
	\centering
	\begin{tikzpicture}[>=latex,
	roundnode/.style={circle, draw=black!90, very thick, minimum size=6mm, inner sep=0pt},
	squarednode/.style={rectangle, draw=red!60, very thick, minimum size=6mm, inner sep=0pt},
	] 
	\node [roundnode] (v) at (4,1){$v$};
	\node [roundnode] (v0) at (3,1){$v_0$};
	
	\node [roundnode] (v1) at (2,1.5){$v_1$};
	\node [squarednode] (v1') at (1,1.5){$v'_1$};
	\node [squarednode] (v2) at (2,0.5){$v_2$};
	\draw [line width =1.4pt, gray, -] (v2)--(v0)--(v); 
	
	\draw [line width =1.4pt, gray, -] (v1')--(v1)--(v0);
	
	\end{tikzpicture}
	\caption{Configuration from Lemma~\ref{lem:3-1}}
	\label{fig:3-1}
\end{figure}
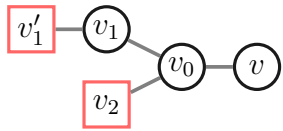

\section{List homomorphism of signed graphs to $(K_6, M)$} \label{sec:precoloring}

In this section, we study the ${\rm DSG}(K_6, M)$-list homomorphism of a given signed graph and develop tools of independent interest that will be used in Section~\ref{sec:K6}. 

In the following lemma, for a given edge $xy$, viewed as a rooted tree with $y$ being the root and $x$ being the leaf, we would like to bound the size of $F_x(y)$ in terms of the size of $L(x)$. Note that the larger the set $L(x)$ is, the smaller the set $F_x(y)$ becomes. Often we will use the following lemma on an edge $xy$ to evaluate $F_x(y)$ without explicitly stating that the edge $xy$ under consideration forms a tree with leaf $x$ and root $y$.

\begin{lemma}
	\label{oneEdgeRestriction}
	Let $xy$ be a signed edge and let $L$ be its ${\rm DSG}(K_6, M)$-list assignment. Then the following statements hold:
	\begin{enumerate}[leftmargin =3em, label=(\arabic*)]
	    \item \label{P2-0} If $L(x)$ is empty, then $F_x(y)$ is $C$.
	       
		\item \label{P2-0.5} If $L(x)$ consists of one element, then $F_x(y)$ is a paired $7$-set.
		
		\item \label{P2-1} If $L(x)$ is a paired $2$-set, then $F_x(y)$ is a paired $6$-set.
		
		\item \label{P2-1.5} If $L(x)$ is a paired $3$-set, then $F_x(y)$ is a paired set of size at most $4$. Moreover, if $L(x)$ is not one-sided, then $|(L(y)\setminus F_x(y))\cap C^+| \geq 4$ and $|(L(y)\setminus F_x(y))\cap C^-| \geq 4$. 
					
		\item \label{P2-2} If $L(x)$ is a paired $4$-set, then $F_x(y)$ is a paired set of size at most $4$. In particular, if $L(x)$ is not layered, then $F_x(y)=\emptyset$, and if $L(x)$ is one-sided, then $F_x(y)$ is a paired $2$-set.
			
		\item \label{P2-2.5} If $L(x)$ is a neighbored $5$-set, then $F_x(y)$ is a paired $2$-set.
		
		\item \label{P2-3} If $L(x)$ is a paired $6$-set, then $F_x(y)$ is a paired set of size at most $2$. In particular, if $L(x)$ is not layered or one-sided, then $F_x(y)=\emptyset$.
		
		\item \label{P2-4} If $L(x)$ is a paired $8$-set, then $F_x(y)=\emptyset$. 
	\end{enumerate}
\end{lemma}

\begin{proof}
     We assume that $xy$ is a positive edge. 
	
	In the first claim, $L(x)=\emptyset$ means that there is no valid color for $x$, in other words, all colors are forbidden at $y$. In the second claim, we may assume $L(x)=\{1^+\}$, then $F_x(y) = \{1^+,2^+,1^-,$ $3^-,4^-,5^-,6^-\}$ which is of size seven. For the third claim, without loss of generality, we may assume our paired $2$-set $L(x) = \{1^+,2^+\}$, then $F_x(y) = \{1^+,2^+,3^-,4^-,5^-,6^-\}$. 
	
	Next, suppose that $L(x)$ is a paired $3$-set. Without loss of generality, we only consider three possibilities $L(x) = \{1^+,2^+,3^+\}$, $L(x) = \{1^+,2^+,1^-\}$, and $L(x) = \{1^+,2^+,3^-\}$. We easily obtain that for the first possibility $F_x(y)=\{3^-,5^-,6^-\}$, for the second possibility $F_x(y)=\{1^+\}$, and for the last possibility $F_x(y)=\{1^+,2^+,3^-,4^-\}$. Observe that $|(L(y)\setminus F_x(y))\cap C^+| \geq 4$ and $|(L(y)\setminus F_x(y))\cap C^-| \geq 4$ when $L(x)$ is not one-sided.

	Suppose now $L(x)$ is a paired set of size $4$. Again without loss of generality, we only need to consider three cases $L(x)= \{1^+,2^+,3^+,4^+\}$ for $L(x)$ being one-sided, $L(x) = \{1^+,2^+,1^-,2^-\}$ for $L(x)$ being not layered, and $L(x) = \{1^+,2^+,3^-,4^-\}$ for $L(x)$ being layered but not one-sided. For the first case $F_x(y)=\{5^-,6^-\}$, for the second case $F_x(y)=\emptyset$, and for the last case $F_x(y) = \{1^+,2^+, 3^-,4^-\}$.
	
	Suppose $L(x)$ is a neighbored $5$-set, say $L(x) = \{1^+,2^+,3^+,4^+,5^-\}$, we easily get $F_x(y) = \{5^-,6^-\}$.
	
	Suppose $L(x)$ is a paired set of size $6$. Without loss of generality, we consider three possibilities $L(x) = C^+$ for $L(x)$ being one-sided, $L(x) = \{1^+,2^+,3^+,4^+,1^-,2^-\}$ for $L(x)$ being not layered, and $L(x) = \{1^+,2^+,3^+,4^+,5^-,6^-\}$ for $L(x)$ being layered but not one-sided. For the first two possibilities $F_x(y)= \emptyset$, and for the last possibility $F_x(y)= \{5^-,6^-\}$.
	
	It is an easy observation that if $L(x)$ contains four elements of a layer, then $F_x(y)=\emptyset$. This implies that $F_x(y)=\emptyset$ when $L(x)$ is a paired set of size at least $8$.
\end{proof}

The next observation follows from Lemma~\ref{oneEdgeRestriction}~\ref{P2-0.5} and \ref{P2-1.5} easily. 

\begin{observation}\label{obs:8-set}
Let $xy$ be a signed edge. If $L(x)$ consists of one color, then $C\setminus F_x(y)$ is a neighbored $5$-set. If $L(x)$ is a neighbored $3$-set, then $C\setminus F_x(y)$ is a paired $8$-set containing neither $C^+$ nor $C^-$.
\end{observation}

The next two observations are also straightforward and we will use them frequently in the Section~\ref{sec:K6}. Note that a paired $8$-set, where all the elements compose two layers, do not contain a neighbored $5$-set.

\begin{observation}\label{obs:K2DifferentLayers}
Let $(K_2, \sigma)$ be a signed edge $xy$. Suppose that $L$ is a list assignment of $(K_2, \sigma)$ satisfying that each of $L(x)$ and $L(y)$ is either a neighbored $5$-set or a paired $8$-set. Then one can choose $c_x \in L(x)$ and $c_y \in L(y)$ such that $c_x$ and $c_y$ are in different layers and $m^*(c_xc_y)=\sigma(xy)$.
\end{observation}

\begin{observation}\label{obs:3path}
Let $(P_3,\sigma)$ be a signed path $xzy$ and let $L$ be a list assignment where $L(z)=C$, $L(x)=\{c_x\}$ and $L(y) = \{c_y\}$. Then $(P_3,\sigma)$ is $L$-colorable unless one of the followings holds:
\begin{enumerate}[leftmargin =3em, label=(\arabic*)]
	\item $c_x$ and $c_y$ are in a same layer but on different sides, and $P_3$ is a positive path;
	\item $c_x$ and $c_y$ are in a same layer and on a same side, and $P_3$ is a negative path.
\end{enumerate}
\end{observation}

Next we have a few lemmas on list coloring of paths and cycles.

\begin{lemma}\label{lem:P_3with12,10,8}
Let $(P_3,\sigma)$ be a signed path $xyz$ and let $L$ be its list assignment satisfying one of the followings:
\begin{enumerate}[leftmargin =3em, label=(\arabic*)]
	\item $L(y)$ is a full list, $L(x)=\{c_x\}$ and $L(z)=\{c_z\}$ where $c_x$ and $c_z$ are in different layers;
	\item $L(y)$ is a paired $10$-set, $L(x)=\{c_x\}$ and $L(z)$ is a neighbored $5$-set;
	\item $|L(y)|\geq 5$ and each of $L(x)$ and $L(z)$ is a neighbored $5$-set;
\end{enumerate}
Then $(P_3,\sigma)$ is $L$-colorable.
\end{lemma}

\begin{proof}
The first case is just a restatement of Observation~\ref{obs:3path}.
To prove the other two cases, by Observation~\ref{obs:L-Lx}, we may assume both edges are positive. For the second case, without loss of generality, we may assume that $c_x=1^+$. If $L(z)$ contains one of $3^-, 4^-, 5^-, 6^-$, say $5^-$, without loss of generality, then we take $c_z=5^-$. Since $L(y)$ is a paired $10$-set, it contains one of $2^-$ or $6^+$ either of which completes the coloring. If $L(z)$ contains none of $3^-, 4^-, 5^-, 6^-$, then, as it is a neighbored $5$-set, it must contain $3^+, 4^+, 5^+, 6^+$. If the choice of $c_z=3^+$ does not work, then $L(y)$ is the complement of $\{5^+, 6^+\}$, in which case taking $c_z=5^+$ would work.

For the third claim, without loss of generality, we assume $L(x)=\{1^+, 2^+, 3^+, 4^+, 5^-\}$. Then except for $5^-$ and $6^-$ every vertex of ${\rm DSG}(K_6,M)$ is connected by a positive edge to at least one vertex in $L(x)$. Similarly, there is only one pair, say $a,b$, of vertices of ${\rm DSG}(K_6,M)$ which is not connected by a positive edge to at least one vertex in $L(z)$. Since $|L(y)|\geq 5$, there is a color in $L(y)$ different from $5^-, 6^-,a,b$. Assigning this color to $y$, we can find choices for $x$ and $z$.
\end{proof}

\begin{lemma}\label{2.5-5}
     Let $(K_2, \sigma)$ be a signed edge $uv$ and let $L$ be its list assignment where $L(u)$ is a neighbored $5$-set and $L(v)$ is either $C^+$ or $C^-$. Then there exists a choice $c_u \in L(u)$ and a $4$-subset $L'(v)$ of $L(v)$ such that for every $c_v \in L'(v)$, $m^*(c_uc_v)=\sigma(uv)$ .
\end{lemma}

\begin{proof}
    As a neighbored $5$-set intersecting both $C^+$ and $C^-$, if $\sigma(uv)$ is positive, then we choose $c_u\in L(u)\cap L(v)$, otherwise we choose $c_u\in L(u)\cap (C\setminus L(v))$. Then $c_u$ has four neighbors in $L(v)$ which are adjacent to it by edges of sign $\sigma(uv)$.
\end{proof}

\begin{lemma}\label{outdegree2restriction}
	Let $(P_3,\sigma)$ be a signed path $xzy$ with $\sigma(xz)=\alpha$ and $\sigma(zy)=\beta$. Given a layered $6$-set $X$, for every $c_x \in X$ and $c_y \in X^{\alpha \beta}$, there exists $c_{z} \in C$ such that $m^*(c_zc_x)=\sigma(zx)$ and $m^*(c_zc_y)=\sigma(zy)$.
\end{lemma}

\begin{proof}
    By Observation~\ref{obs:L-Lx}, we assume that $\alpha=\beta=+$. Then $X^{\alpha\beta}=X$. For every $c_x, c_y\in X$, since $X$ is a layered $6$-set, either (1) $c_x=$ Pair$(c_y)$ or (2) $c_x$ and $c_y$ are in different layers. In both cases, Observation~\ref{obs:3path} assures that $c_z$ exists.
\end{proof}

\begin{lemma}\label{5-10path}
Let $(P_{2k}, \sigma)$ be a signed path where $P_{2k}=v_1v_2\cdots v_{2k}$ for $k\geq 1$ and let $L$ be a list assignment of $(P_{2k}, \sigma)$ satisfying one of the followings:
	\begin{enumerate}[leftmargin =3em, label=(\arabic*)]
		\item \label{3-10-5} $L(v_1)$ is a paired $3$-set, $L(v_{2k})$ is a one-sided $4$-set, and for $i \in \{2, \ldots, 2k-1\}$, $L(v_i)$ contains a neighbored $5$-set for odd $i$ and $|L(v_i)|\geq 10$ for even $i$. 
		\item \label{4-10-5} $L(v_{1})$ is a one-sided $4$-set, $L(v_{2k})$ is a paired $3$-set and for all $i \in \{2, \ldots, 2k-1\}$, $L(v_i)$ contains a neighbored $5$-set for odd $i$ and $|L(v_i)|\geq 10$ for even $i$. 
	\end{enumerate}
Then $(P_{2k}, \sigma)$ is $L$-colorable.
\end{lemma}

\begin{proof}
For $k=1$, we consider a signed edge $v_1v_2$ with $v_1$ being its root. Let $L$ be a list assignment satisfying that $L(v_1)$ is a paired $3$-set and $L(v_2)$ is a one-sided $4$-set. By Lemma~\ref{oneEdgeRestriction}~\ref{P2-2}, $|F_{v_2}(v_1)|=2$. For any color $c\in L(v_1)$, we can color $v_1$ by the remaining color from $L(v_1)\setminus F_{v_2}(v_1)$.
	
For the first case with $k\geq 2$, consider $(P_{2k}, \sigma)$ as a rooted signed tree with root $v_{2k-1}$. Since $|L(v_1)|\geq 3$, $L^a(v_2)$ which is $L(v_2)\setminus F_{v_1}(v_2)$, by Lemma \ref{oneEdgeRestriction} \ref{P2-1.5}, contains a paired set of size at least $6$. By Lemma \ref{oneEdgeRestriction} \ref{P2-3} and Lemma \ref{oneEdgeRestriction} \ref{P2-1.5}, we can propagate this and obtain that $|L^a(v_{2i-1})| = |L(v_{2i-1}) \setminus F_{v_{2i-2}}(v_{2i-1})| \geq 5-2= 3$ and $|L^a(v_{2i})|=|L(v_{2i})\setminus F_{v_{2i-1}}(v_{2i})| \geq 10-4= 6$ for $i\in\{2,\ldots,k-2\}$. 
Finally, by Lemmas \ref{oneEdgeRestriction} \ref{P2-2} and \ref{oneEdgeRestriction} \ref{P2-3}, for the root $v_{2k-1}$, we have $|L^a(v_{2k-1})|=|L(v_{2k-1}) \setminus (F_{v_{2k}}(v_{2k-1}) \cup F_{v_{2k-2}}(v_{2k-1}))|\geq 5-2-2=1$. 

For the second case, we similarly consider $(P_{2k}, \sigma)$ as a rooted signed tree with root $v_{2k-1}$. So $|L^a(v_2)|=|L(v_2)\setminus F_{v_{1}}(v_2)| \geq 10-2= 8$, we can propagate this and obtain that $|L^a(v_{2i})|=|L(v_{2i})\setminus F_{v_{2i-1}}(v_{2i})| \geq 10-2= 8$ and $|L^a(v_{2i+1})| = |L(v_{2i+1}) \setminus F_{v_{2i}}(v_{2i+1})| \geq 5-0= 5$ for $i\in\{2,\ldots,k-2\}$. For the root $v_{2k-1}$, we have $|L^a(v_{2k-1})|=|L(v_{2k-1}) \setminus (F_{v_{2k}}(v_{2k-1}) \cup F_{v_{2k-2}}(v_{2k-1}))|\geq 5-4-0=1$. This completes the proof.
\end{proof}

\begin{lemma}\label{lem:List-4-Cycle}
Let $(C_4, \sigma)$ be a signed $4$-cycle $xyzt$ and let $L$ be a list assignment where each of $L(x)$ and $L(t)$ is a neighbored $5$-set, and each of $L(y)$ and $L(z)$ is a paired $10$-set. Then $(C_4, \sigma)$ is $L$-colorable.
\end{lemma}

\begin{proof}
By switching, if necessary, we may assume that $xy$, $yz$ and $zt$ are all positive edges. If, after these switchings, $xt$ is also a positive edge, then we choose colors $c_x\in L(x)$ for $x$ and $c_t\in L(t)$ for $t$ such that they are in different layers but on a same side. This is possible because $L(x)$ and $L(t)$ are both neighbored $5$-sets. We may then assume, without loss of generality, $c_x=1^+$ and $c_t=3^+$. Then on ${\rm DSG}(K_6, M)$ the possible choices for the pair of two vertices $(y, z)$ are: $( 2^-, 4^-)$, $(4^+, 5^+)$ and $(6^+, 2^+)$. But since in each of $L(y)$ and $L(z)$ only one pair is missing, at least one of these three possibilities works.

Thus we may assume $xt$ is a negative edge. We may always have a choice for $c_x\in L(x)$ and $c_t\in L(t)$ such that they are on different sides and in different layers. Without loss of generality, assume $c_x=1^+$ and $c_t=3^-$. Then the option to extend $xyzt$-path is either coloring vertex $y$ with $2^-$ and vertex $z$ from $\{5^-, 6^-\}$, or coloring vertex $z$ with $4^+$ and vertex $yc$ from $\{5^+, 6^+\}$. Therefore, if $(C_4, \sigma)$ is not $L$-colorable, then either $L(y)=\{1^-,2^-\}^c, L(z)=\{3^+,4^+\}^c$ or $L(y)=\{5^+, 6^+\}^c, L(z)=\{5^-, 6^-\}^c$. However, as $L(x)$ and $L(t)$ are both neighbored $5$-sets, we also have one of the two following possibilities for $c_x$ and $c_t$: either they form a pair, or $c_x\in C^-$ and $c_t\in C^+$. Based on whichever the possibility, we will have a choice for $y$ and $z$ this time.
\end{proof}

\begin{lemma}\label{lem:evencycle}
Let $(C_{2k},\sigma)$ be a signed even cycle $v_1v_2\cdots v_{2k}$ with $k\geq 2$ and let $L$ be a list assignment of $(C_{2k}, \sigma)$ satisfying one of the followings:
	\begin{enumerate}[leftmargin =3em, label=(\arabic*)]
		\item \label{even1} For even values of $i$, $L(v_i)$ is a neighbored $5$-set, and for odd values, it is a paired $10$-set.  
		\item \label{even2} $L(v_1)=C$, $L(v_2)$ is a paired $8$-set having four elements on each side, and $L(v_{2k})$ is a neighbored $5$-set. For other vertices $v_i$, if $i$ is even then $L(v_i)$ is a paired $10$-set and if $i$ is odd then $L(v_i)$ is a neighbored $5$-set. 

		\item \label{even3} $L(v_1)=C$, $L(v_2)$ and $L(v_{2k})$ are neighbored $5$-sets, $L(v_3)$ is a paired $8$-set having four elements on each side. For other vertices, if any, $L(v_i)$ is a neighbored $5$-set if $i$ is odd and $L(v_i)$ is a paired $10$-set otherwise. 
	\end{enumerate}
Then $(C_{2k}, \sigma)$ is $L$-colorable.
\end{lemma}

\begin{proof}
	 By Observation~\ref{obs:L-Lx}, we may assume that $\sigma(v_{2k}v_1)=\sigma(v_2v_1)=\sigma(v_{2}v_{3})=+$ in all cases. We consider each case separately.

      \ref{even1}. Since $L(v_{2k})$ is a neighbored $5$-set, either $L(v_{2k})\cap C^+$ or $L(v_{2k})\cap C^-$ is a one-sided $4$-set. Without loss of generality, we assume that $|L(v_{2k})\cap C^+|=4$. Since $L(v_2)$ is also a neighbored $5$-set, $L(v_2)$ has at least one element in $C^+$, without loss of generality, let $1^+$ be one such an element and assign it to $v_2$. Then, using Lemma~\ref{oneEdgeRestriction}~\ref{P2-0.5}, update the list of colors available at $v_3$ to a neighbored $3$-set $L'(v_3)$. Furthermore, update the list of available colors at $v_{2k}$ to the one-sided 4-set $L'(v_{2k})=C^+\cap L(v_{2k})$ and leave the lists of other vertices as they were given. Applying Lemma~\ref{5-10path}~\ref{4-10-5} to the signed path $(P_{2k-2},\sigma)$ with $P_{2k-2}=v_3\cdots v_{2k}$ and with the modified list assignment given above, we color all the vertices of $P_{2k-2}$. If $v_{2k}$ is colored $1^+$ or $2^+$, then, as $L(v_1)$ is a 10-set, we will find a choice to extend the coloring to $v_1$ and we are done. Else, by symmetry among colors $3^+,4^+,5^+,6^+$, we may assume $v_{2k}$ colored $3^+$. If this coloring is not extendable to $v_1$, then $L(v_1)=C-\{5^+,6^+\}$. 

      If $\{5^+,6^+\}\cap L(v_2) \neq \emptyset$, then by choosing a color for $v_2$ from $\{5^+,6^+\}$ and repeating the same process, this time we are sure to have a choice to color $v_1$. Thus we have one of the two possibilities for $L(v_2)$: $i$. It contains $\{1^+,2^+,3^+,4^+\}$. $ii$. It is $\{1^+,3^-,4^-, 5^-,6^-\}$. If the latter, we choose a color for $L(v_{2k})$ from $C^-$, and repeat the previous process with $L'(v_{2k-1})$ being a neighbored 3-set and $L'(v_2)=\{3^-,4^-, 5^-,6^-\}$. Hence, we must have $\{1^+,2^+,3^+,4^+\}\subset L(v_2)$. By the symmetry of $v_{2k}$ and $v_{2}$, we also have $\{1^+,2^+,3^+,4^+\}\subset L(v_{2k})$. But then once again we repeat the original process by assigning the color in $L(v_2)\cap \{5^-,6^-\}$ to $v_2$ and taking  $L'(v_{2k})=\{1^+,2^+,3^+,4^+\}$. This time for each choice of color for $v_{2k}$ we will find a choice for $v_{1}$ which is not in $\{5^+,6^+\}$.
  
      \ref{even2}.  Since $L(v_{2k})$ is a neighbored $5$-set, without loss of generality, we assume that $|L(v_{2k})\cap C^+|=4$. Let $X=C^+$ and let $L'(v_2)=L(v_2)\cap X$. We update the list of available colors at $v_2$ to $L'(v_2)$, observing that it is a one-sided $4$-set. Considering the tree $v_2v_3$ rooted at $v_3$, by Lemma~\ref{oneEdgeRestriction}~\ref{P2-2}, since $L(v_3)$ is a neighbored $5$-set, we update the list of available colors at $v_3$ to a neighbored $3$-set $L'(v_3)$. Now set $L'(v_{2k})=X \cap L(v_{2k})$ and $L'(v_i)=L(v_i)$ for $i\in \{4, \ldots, 2k-1\}$. Applying Lemma~\ref{5-10path}~\ref{3-10-5} to the signed path $P_{2k-2}=v_3\cdots v_{2k}$ and with the modified list assignment given here, we color all the vertices of $P_{2k-2}$. Noting that the colors of $v_{2k}$ and $v_2$ are both chosen from $X$, and by Lemma~\ref{outdegree2restriction}, we can complete the coloring to $v_1$.

	   \ref{even3}. We first assume that for a color $c\in L(v_2)$, four positive neighbors (connected by positive edges) of $c$ are in $L(v_3)$. In that case, we color $v_2$ by $c$. We update $L(v_{2k})$ by removing colors in the same layer as $c$ and $L(v_3)$ by taking the four positive neighbors of $c$. To complete the coloring, we color the path $v_3 v_4\cdots v_{2k}$ by Lemma~\ref{5-10path}. Then by Lemma~\ref{lem:P_3with12,10,8}, we can find a color for $v_1$ and we are done. 

       If no such a choice of $c$ exists, then $L(v_2)\subset L(v_3)$, as otherwise any color in $L(v_2)\setminus L(v_3)$ would work. Thus without loss of generality, we assume $L(v_2)=\{1^+, 2^+,3^+,4^+,5^-\}$ and $L(v_3)=\{1^+, 2^+,3^+,4^+,3^-,4^-,5^-,6^-\}$. Next, we examine the choice of $3^+$ for $v_2$. This updates $L(v_3)$ to a neighbored 3-set. If $L(v_{2k})$ contains at most one of $3^-$ and $4^-$, then the updated list $L'(v_{2k})=L(v_{2k})\setminus \{3^-, 4^-\}$ contains a one-sided $4$-set and we are done by applying Lemma~\ref{5-10path}. Thus we assume $\{3^-, 4^-\}\subset L(v_{2k})$, this implies that $|L(v_{2k})\cap \{1^+,2^+, 5^+,6^+\}|=1$, let $c_1$ be the common element. We consider two cases:
    
        \begin{itemize}
        \item $c_1\in \{1^+,2^+\}$. Then we set  $L'(v_{2k})=\{3^-,4^-, 5^-,6^-\}$ and, by using Lemma~\ref{5-10path}, we color the path $v_4\cdots v_{2k}$. Then depending on the color of $v_4$, we choose a color for $v_3$ from $\{3^+,4^+,3^-,4^-\}$. Based on this choice, we color $v_2$ from $\{1^+, 2^+,5^-\}$. By Observation~\ref{obs:3path}, since $v_{2k}v_1v_2$ is a positive path, any of these choices for $v_{2k}$ and $v_2$ can be extended to $v_1$.

        \item $c_1\in \{5^+,6^+\}$. We set $L'(v_{2k})=\{3^-,4^-, c_1\}$, and consider the path $v_4\cdots v_{2k}$ (when $k=2$, this is a single vertex). By Lemma~\ref{5-10path} and with the list assignment $L(v_4),\ldots, L(v_{2k-1}), L'(v_{2k})$, this path admits a coloring. Let $c_2$ be the color of $v_{4}$. Of the elements in $\{3^+, 4^+, 3^-, 4^-\}$ at least one, say $c_3$, is adjacent to $c_2$ by a positive edge. Assign $c_3$ to $v_3$. If $c_3 \in \{3^+, 4^+\}$, then we choose a color from $\{1^+, 2^+\}$ for $v_2$; if $c_3\in \{3^-, 4^-\}$, then there exists a color from $\{3^+, 4^+, 5^-\}$ for $v_2$. In each case, we can choose a color for $v_2$ which is not in the same layer as the color of $v_{2k}$. By Lemma~\ref{lem:P_3with12,10,8}, we can extend this coloring to $v_1$ and complete the coloring of this even cycle.
    \end{itemize}

Thus in all cases we find a valid coloring.
\end{proof}

\begin{lemma}
	\label{lem:oddcycle}
	Let $(C_{2k+1},\sigma)$ be a signed odd cycle where $C_{2k+1}=v_1v_2\cdots v_{2k+1}$ and let $L$ be a list assignment of $(C_{2k+1}, \sigma)$ satisfying one of the followings:
	\begin{enumerate}[leftmargin =3em, label=(\arabic*)]
		\item \label{odd1} $L(v_{i})$ is a neighbored $5$-set for each even $i$ and $L(v_i)$ is a paired $10$-set for each odd $i$. 
		\item \label{odd2} $L(v_1)=C$, and $L(v_{2k+1})$ is a neighbored $5$-set. For other vertices, $L(v_i)$ is a neighbored $5$-set if $i$ is even and $L(v_i)$ is a paired $10$-set otherwise.  
	\end{enumerate} 
Then $(C_{2k+1}, \sigma)$ is $L$-colorable. 
\end{lemma}

\begin{proof}
	\ref{odd1}. Applying Lemma \ref{2.5-5} to the signed edge $v_{2k}v_{2k+1}$, since $L(v_{2k})$ is a neighbored $5$-set and $L(v_{2k+1})$ is a paired $10$-set (containing $C^+$ or $C^-$), we can assign a color $c_{v_{2k}}\in L(v_{2k})$ to $v_{2k}$ and choose a subset $L'(v_{2k+1})$ of $L(v_{2k+1})$ which is a one-sided $4$-set and have the property that for each $c\in L'(v_{2k+1})$ we have $m^*(c_{v_{2k}}c)=\sigma(v_{2k}v_{2k+1})$. Considering the signed edge $v_{2k}v_{2k-1}$, by Lemma \ref{oneEdgeRestriction}~\ref{P2-0.5}, $|L^a(v_{2k-1})| \geq 3$. We set $L'(v_{2k-1}):=L^a(v_{2k-1})$ and $L'(v_i)=L(v_i)$ for $i \in\{1,2, \ldots, 2k-2\}$. We may now apply Lemma~\ref{5-10path}~\ref{4-10-5} to the signed path $(P_{2k}, \sigma)$ where $P_{2k}=v_{2k+1}v_1v_2\cdots v_{2k-1}$ and we are done.
	
	\ref{odd2}. The proof of this case is similar as the proof of Lemma~\ref{lem:evencycle}~\ref{even3}. For $k=1$, by Observation \ref{obs:K2DifferentLayers}, we can always choose $c_{v_1}\in L(v_1)$ and $c_{v_3}\in L(v_3)$ which are in different layers such that the sign of $v_1v_3$ is preserved. By Lemma~\ref{lem:P_3with12,10,8}, this coloring can be extended to $v_2$. Thus we may assume $k\geq 2$ and, by Observation~\ref{obs:L-Lx}, we may assume that $\sigma(v_1v_2)= \sigma(v_1v_{2k+1})= \sigma(v_{2}v_{3})=+$. Furthermore, without loss of generality, we assume that $L(v_{2})=\{1^+,2^+,3^+,4^+,5^-\}$. 

    If $L(v_{2k+1})$ contains four elements of $C^+$, then after taking three of them as $L'(v_{2k+1})$ and setting $L'(v_2)=\{1^+,2^+,3^+,4^+\}$ while keeping the rest of the lists same, we may apply Lemma~\ref{5-10path} to color the path $v_2v_3\cdots v_{2k+1}$. This coloring then is extendable to $v_1$ by Lemma~\ref{outdegree2restriction}. 

    If $L(v_{2k+1})$ contains $\{5^-,6^-\}$, then we may take $L'(v_{2k+1})$ to consist of $5^-, 6^-$ and the only element of $L(v_{2k+1})$ in $C^+$, and complete the coloring as in the previous case. We may, therefore, assume $L(v_{2k+1})= \{1^-,2^-,3^-,4^-, c\}$ where $c\in \{5^+,6^+\}$.

    Next we consider the case that $\{1^-,2^-,3^-,4^-\}\subset L(v_3)$. By Lemma~\ref{oneEdgeRestriction}~\ref{P2-2}, there is a $3$-subset $L'(v_{2k})$ of $L(v_{2k})$ such that the signed edge $v_{2k}v_{2k+1}$ is $L'$-colorable for every choice $c_{v_{2k}} \in L'(v_{2k})$ and $L'(v_{2k+1})=\{1^-,2^-,3^-,4^-\}$. We may now first color the path $v_3v_4\cdots v_{2k}$ with list assignment $L'(v_3)=\{1^-,2^-,3^-,4^-\}$, $L'(v_{2k})$ defined above, and $L(v_i)$ for all other values of $i$ by Lemma~\ref{5-10path}. To complete the coloring, then we color $v_2$ with $5^-$, and $v_{2k+1}$ with a color from $\{1^-,2^-,3^-,4^-\}$. This coloring then is easily extendable to $v_1$. This, without loss of generality, we may assume that $L(v_3)=C\setminus \{1^-, 2^-\}$. 

   To complete the proof, we consider the list assignment on the path $v_4v_5\cdots v_{2k+1}$ where $v_{2k+1}$ is assigned $L'(v_{2k+1})=\{3^+, 4^+, 5^-\}$ and each other vertex is assigned $L(v_i)$. By Lemma~\ref{5-10path}, we have a list-coloring $\phi$ of this path. Given $\phi(v_4)$, one can always pick a color in $\{3^+, 4^+, 3^-, 4^-\}$ for $v_3$ such that the sign of the edge $v_3v_4$ is preserved. However, for any such choice, the coloring of $v_3$ and $v_{2k+1}$ can be extended to $v_1$ and $v_2$ by argument similar to the end of the proof of previous lemma. This concludes the proof.
\end{proof}

\section{Mapping signed graphs to $(K_{6},M)$}\label{sec:K6}

In order to prove Theorem~\ref{main-theorem}, it suffices to prove the following theorem.

\begin{theorem}\label{main-theorem2}
Let $(G, \sigma)$ be a signed graph with $\mad(G)<\frac{14}{5}$. Then $(G, \sigma)$ admits an edge-sign preserving homomorphism to ${\rm DSG}(K_6, M)$.
\end{theorem} 

To prove Theorem~\ref{main-theorem2}, we assume that $(G, \sigma)$ is a minimum counterexample, i.e. $\mad(G) < \frac{14}{5}$, it does not map to ${\rm DSG}(K_6, M)$ and it has as small as possible number of vertices. The proof is organized as follows. In Subsection~\ref{sec:forbidden-configurations}, we give a set of reducible configurations in a minimum counterexample $(G, \sigma)$. In Subsection~\ref{sec:discharging}, we use discharging arguments to show that at least one reducible configuration listed in Subsection~\ref{sec:forbidden-configurations} exists in $(G, \sigma)$, which is a contradiction and completes the proof.

\subsection{Reducible configurations}\label{sec:forbidden-configurations}

To better state the forbidden configuration, we use the following standard terminology: a vertex of degree $k$ maybe referred to as a $k$-vertex. Moreover, a \emph{$k^+$-vertex} is a vertex with degree at least $k$ and a \emph{$k^-$-vertex} is a vertex of degree at most $k$. A \emph{$k_i$-vertex} is a $k$-vertex with precisely $i$ neighbors of degree $2$. When proving that a configuration $F$ is forbidden, we consider $F$ together with all its neighbors that are precolored. A precolored neighbor, say $v$, of $F$ might see more than one vertex in $F$, however, for simplicity we will view such a configuration with multiple copies of $v$, one for each neighbor in $F$, and where all copies are colored the same as $v$. In special case that $F$ is a tree, this will allow us to view the subgraph induced by $F$ and its neighbors as a tree.

As $(K_6, M)$ is a vertex-transitive signed graph, equivalently, ${\rm DSG}(K_6,M)$ is vertex-transitive as a $2$-edge-colored graph, we have the following.

\begin{lemma}\label{lem:2-connected} The graph $G$ is $2$-connected, in particular, we have $\delta(G) \geq 2$.
\end{lemma}

Next we show that vertices of certain types are reducible.

\begin{lemma}\label{lem:weakvertices}
The graph $G$ does not contain the following vertices: $2_1$-vertex, $3_2$-vertex,  $4_4$-vertex, $5_5$-vertex.
\end{lemma}
\begin{proof}
	Observe that a $2_1$-vertex (respectively, $4_4$-vertex) is a subcase of a $3_2$-vertex (respectively, $5_5$-vertex). Thus it is enough to prove that $G$ has no $3_2$-vertex or $5_5$-vertex.  
	
		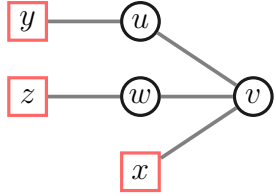
\begin{figure}[h]
		\centering
		\begin{tikzpicture}[>=latex,
		roundnode/.style={circle, draw=black!90, very thick, minimum size=5mm, inner sep=0pt},
		roundnode2/.style={circle, draw=black!90, very thick, minimum size=5mm, inner sep=0pt},
		squarednode/.style={rectangle, draw=red!60, very thick, minimum size=5mm, inner sep=0pt},
		] 
		\node [roundnode2](v) at (3,2){$v$};
		\node [roundnode](u) at (1.5,3){$u$};
		\node [roundnode](w) at (1.5,2){$w$};
		\node [squarednode](x) at (1.5,1){$x$};
		\node [squarednode](y) at (0,3){$y$};
		\node [squarednode](z) at (0,2){$z$};
		\draw [line width =1.4pt, gray, -] (y)--(u)--(v);
		\draw [line width =1.4pt, gray, -] (z)--(w)--(v);
		\draw [line width =1.4pt, gray] (x)--(v);
		\end{tikzpicture}
		\caption{Reducible configuration 3: $3_2$-vertex}
		\label{fig:FC3}
	\end{figure}

	Let $v$ be a $3_2$-vertex, let $u$ and $w$ be its $2$-neighbors and let $x$ be its third neighbor. Moreover, let $y$ and $z$ be the other neighbors of $u$ and $w$ (see Figure~\ref{fig:FC3}). We claim that any $L$-coloring of $x$, $y$ and $z$ could be extended to a coloring of the signed tree induced by $u$, $v$ and $w$ (rooted at $v$). That is because the color of $x$ reduces the list of available colors at $x$ to a neighbored $5$-set of which at most two elements become forbidden respectively by each of the $vu$ and $vw$-branches, leaving at least one admissible color at $v$. 
	
	The proof of $5_5$-vertex is similar. Let $v$ be a $5_5$-vertex. Consider the tree $T$ induced by $v$, its neighbors and their neighbors and suppose it is rooted at $v$. Assume all the leaves of $T$ are colored. Thus we have a full list of colors available at $v$ at the start, each of the five branches may forbid two from $L(v)$, leavings us with at least two colors in the admissible set of $v$.   
\end{proof}

Observe that this simple argument will not work on other type of vertices, more precisely for any other type of a vertex $v$, there will be a coloring of its neighborhood that leaves us with an empty admissible set at $v$. However, in the next series of lemmas, we put some restriction on the neighborhood of vertices of type $3_1$, $3_0$ and $4_3$.

\begin{lemma}\label{lem:RC4_3-3_1}
	In the graph $G$, no $3_1$-vertex is adjacent to a $4_3$-vertex.
\end{lemma}

\begin{proof}
	Suppose to the contrary that a $3_1$-vertex $u$ is adjacent to a $4_3$-vertex $v$. 
	Let $u_1, u_2$ be the other two neighbors of $u$ with $u_2$ being of degree 2, and let $v_1,v_2,v_3$ be the other three neighbors of $v$ all of which are all of degree 2.
	We first consider the case that $u_2$ is distinct from $v_1,v_2,v_3$. 
	Let $w_0$ be the other neighbor of $u_2$, finally let $w_1, w_2, w_3$ be the other neighbors of $v_1,v_2,v_3$ respectively. Observe that $w_i$ and $u_1$ are not necessarily distinct vertices of $G$, but as they are going to be precolored, for the sake of this proof we may assume that they are distinct. Let $T$ be the tree induced by $u, v, u_2, v_1,v_2, v_3$ and $w_0, w_1,w_2,w_3$ and consider it as rooted at $v$. See Figure~\ref{FC5} for illustration.
	
		\begin{figure}[h]
		\centering
		\begin{tikzpicture}[>=latex,
		roundnode/.style={circle, draw=black!90, very thick, minimum size=5mm, inner sep=0pt},
		roundnode2/.style={circle, draw=black!90, very thick, minimum size=5mm, inner sep=0pt},
		squarednode/.style={rectangle, draw=red!60, very thick, minimum size=5mm, inner sep=0pt}, scale=0.86
		] 
		\node [roundnode2] (u) at (2,1){$u$};
		\node [roundnode2] (v) at (3,1){$v$};
		\node [squarednode] (u1) at (1,0){$u_1$};
		\node [roundnode] (u2) at (1,2){$u_2$};
		\node [roundnode](v2) at (4,1){$v_2$};
		\node [roundnode](v3) at (4,0){$v_3$};
		\node [roundnode](v1) at (4,2){$v_1$};
		\node [squarednode](w0) at (-0.5,2){$w_0$};
		\node [squarednode](w1) at (5.5,2){$w_1$};
		\node [squarednode](w2) at (5.5,1){$w_2$};
		\node [squarednode](w3) at (5.5,0){$w_3$};
		\draw [line width =1.4pt, gray, -] (u)--(u2);
		\draw [line width =1.4pt, gray] (u2)--(w0);
		\draw [line width =1.4pt, gray] (u)--(u1);
		\draw [line width =1.4pt, gray, -] (u)--(v);
		\draw [line width =1.4pt, gray, -] (v)--(v1);
		\draw [line width =1.4pt, gray, -] (v)--(v2);
		\draw [line width =1.4pt, gray, -] (v)--(v3);
		\draw [line width =1.4pt, gray] (v1)--(w1);
		\draw [line width =1.4pt, gray] (v)--(v2)--(w2);
		\draw [line width =1.4pt, gray] (v)--(v3)--(w3);
		\end{tikzpicture}
		\caption{ $u$ and $v$ do not share a common $2$-neighbor}
		\label{FC5}
	\end{figure}
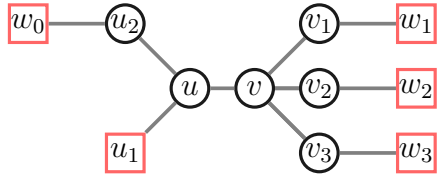

Let  $L$ be a list assignment which assigns a list of size $1$ to each of $w_i$'s and $u_1$, and a full list to the other vertices of $T$. Observe that, by Lemma~\ref{lem:3-1}, the $u$-branch of the tree forbids at most two pairs of colors from $L(v)$ and, by Lemma~\ref{lem:2-Path}, each other branch forbids at most one pair of colors from it, thus $L^a(v)$ contains at least one pair of admissible colors. This completes the proof of this case.

Now we consider the case that $u$ and $v$ have a common $2$-neighbor, say $w$. Let $u_1$ be the other neighbor of $u$ and let $v_1$, $v_2$ be the other two neighbors of $v$. Furthermore let each of $w_1$ and $w_2$ be the neighbor of $v_1$ and $v_2$ (respectively) distinct from $v$. (see Figure~\ref{FC5-2}).

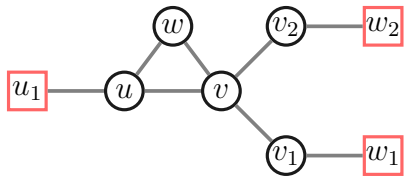
\begin{figure}[h]
	\centering
	\begin{tikzpicture}[>=latex,
	roundnode/.style={circle, draw=black!90, very thick, minimum size=5mm, inner sep=0pt},
	roundnode2/.style={circle, draw=black!90, very thick, minimum size=5mm, inner sep=0pt},
	squarednode/.style={rectangle, draw=red!60, very thick, minimum size=5mm, inner sep=0pt}, scale=0.86
	] 
	\node [roundnode2] (u) at (1.5,1){$u$};
	\node [roundnode2] (v) at (3,1){$v$};
	\node [squarednode] (u1) at (0,1){$u_1$};
	\node [roundnode](v1) at (4,0){$v_1$};
	\node [roundnode](v2) at (4,2){$v_2$};
	\node [squarednode](w1) at (5.5,0){$w_1$};
	\node [squarednode](w2) at (5.5,2){$w_2$};
	\node [roundnode](w) at (2.25,2){$w$}; 
	\draw [line width =1.4pt, gray] (u)--(w)--(v);
	\draw [line width =1.4pt, gray] (u)--(u1);
	\draw [line width =1.4pt, gray] (u)--(v);
	\draw [line width =1.4pt, gray, -] (v)--(v1);
	\draw [line width =1.4pt, gray] (v1)--(w1);
	\draw [line width =1.4pt, gray, -] (v)--(v2);
	\draw [line width =1.4pt, gray] (v2)--(w2);
	\end{tikzpicture}
	\caption{ $u$ and $v$ share a $2$-neighbor}
	\label{FC5-2}
\end{figure}

	As before, we may assume that $u_1$, $w_1$ and $w_2$ (precolored vertices) are distinct. Let $T$ be the  tree induced by $\{u_1, u,v,v_1,v_2, w_1,w_2\}$ and let $L$ be a list assignment which gives a single color to $u_1$, $w_1$ and $w_2$ and a full list to the other vertices. We will show that $T$ has an $L$-coloring such that colors of $u$ and $v$ are in different layers. This completes the proof as for any such choice of colors for $u$ and $v$ one may find an extension for $w$ by Observation~\ref{obs:3path}. Our claim itself is the result of the fact that considering $uu_1$-branch of $T$, $L^a(u)$ is a neighbored $5$-set and considering only $vv_1$ and $vv_2$-branches, $L^a(v)$ contains a paired $8$-set.  
\end{proof}

\begin{lemma}\label{lem:RC3_1-2-3_1}
There are no adjacent $3_1$-vertices who share a common $2$-neighbor in $G$.
\end{lemma}

\begin{proof}
	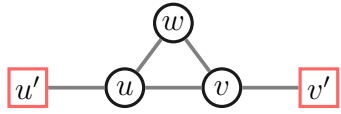
\begin{figure}[h]
		\centering
		\begin{tikzpicture}[
		roundnode/.style={circle, draw=black!90, very thick, minimum size=5mm, inner sep=0pt},
		roundnode2/.style={circle, draw=black!90, very thick, minimum size=5mm, inner sep=0pt},
		squarednode/.style={rectangle, draw=red!60, very thick, minimum size=5mm, inner sep=0pt}, scale=0.86
		] 
		\node [roundnode2] (u) at (1.5,1){$u$};
		\node [roundnode2] (v) at (3,1){$v$}; 
		\node [roundnode] (w) at (2.25,2){$w$}; 
		\node [squarednode] (u') at (0,1){$u'$};
		\node [squarednode](v') at (4.5,1){$v'$};
		 
		\draw [line width =1.4pt, gray] (u)--(u');
		\draw [line width =1.4pt, gray] (u)--(w);
		\draw [line width =1.4pt, gray] (u)--(v); 
		\draw [line width =1.4pt, gray] (v)- -(v');
		\draw [line width =1.4pt, gray] (v)--(w);
		\end{tikzpicture}
		\caption{ Two adjacent $3_1$-vertices share a common $2$-neighbor.}
		\label{FC8}
	\end{figure}

	Assume to the contrary that $u$ and $v$ are two adjacent $3_1$-vertices of $G$, and $w$ is the common $2$-neighbor of them. Let $u'$ be the third neighbor of $u$ and let $v'$ be the third neighbor of $v$, see Figure \ref{FC8}. Thus we have a list assignment on the subgraph induced by $u', u, w, v, v'$, where $u'$ and $v'$ are precolored and the other three have a full list. Our claim then follows from Observations~\ref{obs:K2DifferentLayers} and \ref{obs:3path} as in the proof of the previous lemma. 	
\end{proof}

\begin{lemma}\label{lem:RC3_0-2-3_1}
	A $3_0$-vertex together with two $3_1$-vertices do not induce a triangle in $G$.
\end{lemma}
\begin{proof}
	Suppose that two adjacent $3_1$-vertices $u$ and $v$ share a $3_0$-neighbor $w$. Since $G$ is $2$-connected, $u$ and $v$ do not have a common $2$-neighbor, moreover, let $u'$ and $v'$ be their 2-neighbors respectively. As $w$ has no $2$-neighbor, its third neighbor $w'$ is distinct from $u'$ and $v'$. Furthermore, we label the other neighbors of $u'$ and $v'$ as $u_1$ and $v_1$, respectively, but noting that, as we will consider them to be precolored vertices, they do not need to be distinct from each other or from $w'$ which will also be precolored. Then the set of admissible lists on $u, v, w$, induced by the coloring of $u_1$, $v_1$ and $w'$, would satisfy the conditions of Lemma~\ref{lem:oddcycle}~\ref{odd1} with $k=1$, proving that this configuration is reducible.
\end{proof}

\begin{lemma}\label{lem:RC3_1-3_1-3_1}
No $3_1$-vertex of $G$ has two $3_1$-neighbors.
\end{lemma}

\begin{proof}
Suppose to the contrary that $u, v, w$ are three $3_1$-vertices and $v$ is adjacent to both $u$ and $w$. Let $u_1$, $v_1$ and $w_1$ each be the $2$-neighbor of $u$, $v$ and $w$ respectively. By Lemma~\ref{lem:RC3_1-2-3_1}, we know $v_1$ is distinct from $u_1$ and $w_1$. Furthermore, $u$ and $w$ are not adjacent, as otherwise we will have a sub-configuration of Lemma~\ref{lem:RC3_0-2-3_1}. Depending on whether $u_1$ and $w_1$ are distinct or not, we consider two cases. 

\begin{figure}[h]
	\begin{minipage}[t]{0.5\textwidth}
		\centering
		\begin{tikzpicture}[>=latex,
		roundnode/.style={circle, draw=black!90, very thick, minimum size=5mm, inner sep=0pt},
		roundnode2/.style={circle, draw=black!90, very thick, minimum size=5mm, inner sep=0pt},
		squarednode/.style={rectangle, draw=red!60, very thick, minimum size=5mm, inner sep=0pt}, scale=0.86
		] 
		\node [roundnode2] (u) at (3.5,1){$u$};
		\node [roundnode2] (v) at (5,1){$v$};
		\node [roundnode2] (w) at (6.5,1){$w$};
		\node [roundnode] (u1) at (3.5,0){$u_1$};
		\node [roundnode] (v1) at (5,0){$v_1$};
		\node [roundnode] (w1) at (6.5,0){$w_1$};
		
		\node [squarednode] (u1') at (3.5,-1){$u'_1$};
		\node [squarednode](v1') at (5,-1){$v'_1$};		
		\node [squarednode](w1') at (6.5,-1){$w'_1$};
		
		\node [squarednode] (u2) at (2,1){$u_2$};
		\node [squarednode](w2) at (8,1){$w_2$};
		
		\draw [line width =1.4pt, gray, -] (u2)--(u)--(v); 
		\draw [line width =1.4pt,gray, -] (w2)--(w)--(v);
		\draw [line width =1.4pt,gray, -] (u1')--(u1)--(u);
		\draw [line width =1.4pt,gray, -] (v1')--(v1)--(v);
		\draw [line width =1.4pt,gray, -] (w1')--(w1)--(w);
		\end{tikzpicture}
		\caption{Case: $u_1\neq w_1$}
		\label{fig:FC8-1}
	\end{minipage}
	\begin{minipage}[t]{0.5\textwidth}
		\centering
		\begin{tikzpicture}[>=latex,
		roundnode/.style={circle, draw=black!90, very thick, minimum size=5mm, inner sep=0pt},
		roundnode2/.style={circle, draw=black!90, very thick, minimum size=5mm, inner sep=0pt},
		squarednode/.style={rectangle, draw=red!60, very thick, minimum size=5mm, inner sep=0pt}, scale=0.86
		] 
		\node [roundnode2] (u) at (3.5,1){$u$};
		\node [roundnode2] (v) at (5,1){$v$};
		\node [roundnode2](w) at (6.5,1){$w$};
		\node [roundnode2] (x) at (5,2){$x$};
		\node [roundnode] (v1) at (5,0){$v_1$};
		\node [squarednode](v1') at (5,-1){$v'_1$};		
		\node [squarednode] (u2) at (2,1){$u_2$};
		\node [squarednode](w2) at (8,1){$w_2$};
		\draw [line width =1.4pt, gray] (u2)--(u)--(v)--(w)--(w2);	
		\draw [line width =1.4pt, gray] (u)--(x);
		\draw [line width =1.4pt, gray, -] (v1')--(v1)--(v);
		\draw [line width =1.4pt, gray] (w)--(x);
		\end{tikzpicture}
		\caption{Case: $u_1=w_1$}
		\label{fig:FC8-2}
	\end{minipage}
\end{figure}

 \noindent
{\bfseries Case 1:} $u_1\neq w_1$.  We will use the labeling of vertices near $u, v$ and $w$ as given in Figure~\ref{fig:FC8-1}, noting that $u_2$, $u'_1$, $v'_1$, $w'_1$ and $w_2$ are distinct from $u, v, w$, $u_1$, $v_1$ and $w_1$, but they are not necessarily distinct from each other, however as they are precolored, this would not matter in our proof. 
	
	We consider the rooted tree at $v$ (of Figure~\ref{fig:FC8-1}) whose leaves are precolored and at start all colors are available on each of the internal vertices. Then, by Lemma~\ref{lem:3-1}, each of the $u$-branch and $w$-branch of the tree forbids at most four colors from $L(v)$, while, by Lemma~\ref{lem:2-Path}, the $v_1$-branch of the tree forbids exactly two colors. Thus there are always at least two admissible colors in $L(v)$. 
		
\medskip
 \noindent
{\bfseries Case 2:} $u_1=w_1$. We follow the labeling of Figure~\ref{fig:FC8-2} where this common $2$-neighbor is relabeled as $x$. 

We note again that vertices $u_2$, $w_2$ and $v'_1$ of this figure are distinct from other vertices of the figure but not necessarily distinct from each other. 
We first assign a list to each of the vertices where $u_2$, $w_2$ and $v'_1$ are precolored, and other five vertices each have a full list. Then we update the lists of $u$, $w$ and $v$ according to, respectively, $u_2$, $v'_1$ and $w_2$. In updated lists, $L(u)$ and $L(w)$ each is a neighbored $5$-set, $L(v)$ is a paired $10$-set and $L(x)$ is a full set. Thus we may apply Lemma~\ref{lem:evencycle}~\ref{even1} with $k=2$.	
\end{proof}

\begin{lemma}\label{lem:RC3_0-three-3_1-etc}
Let $u$ be a $3_0$-vertex of $G$ whose neighbors $x$, $y$ and $z$ are all $3_1$-vertices. If $x$ and $y$ have a common neighbor, say $w$, then $d(w)\geq 4$.
\end{lemma}

\begin{proof}
	We first observe that, by Lemma~\ref{lem:RC3_0-2-3_1}, $\{x, y, z\}$ is an independent set of vertices. Let $w$ be a common neighbor of $x$ and $y$. Observe that $w$ is not a $3_1$-vertex as otherwise it would contradict Lemma~\ref{lem:RC3_1-3_1-3_1}. Next we show that $w$ cannot be a $2$-vertex.	
	
	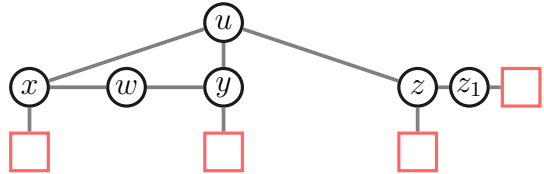
\begin{figure}[h]
		\centering
		\begin{tikzpicture}[>=latex,
		roundnode/.style={circle, draw=black!90, very thick, minimum size=5mm, inner sep=0pt},
		roundnode2/.style={circle, draw=black!90, very thick, minimum size=5mm, inner sep=0pt},
		squarednode/.style={rectangle, draw=red!60, very thick, minimum size=5mm, inner sep=0pt}, scale=0.86
		] 
		\node [roundnode2] (u) at (3,3){$u$};
		\node [roundnode2] (y) at (3,2){$y$};
		\node [roundnode2](w) at (1.5,2){$w$};
		\node [roundnode2] (x) at (0,2){$x$};
		\node [squarednode] (x1) at (0,1){};
		\node [squarednode] (y1) at (3,1){};
		\node [roundnode2] (z) at (6,2){$z$};
		\node [roundnode] (z1) at (6.8,2){$z_1$};
		\node [squarednode] (z3) at (6,1){};
		\node [squarednode] (z2) at (7.6,2){};
		
		\draw [line width =1.4pt, gray] (u)--(z); 
		\draw [line width =1.4pt, gray] (u)--(x)--(w);
		\draw [line width =1.4pt, gray] (x)--(x1);
		\draw [line width =1.4pt, gray] (y)--(y1);
		\draw [line width =1.4pt, gray] (z3)--(z)--(z1)--(z2);
		\draw [line width =1.4pt, gray] (u)--(y)--(w);	
		
		\end{tikzpicture}
		\caption{Neighborhood of a $3_0$-vertex}
		\label{fig:3_03_13_1-2}
	\end{figure}
	Suppose to the contrary that $w$ is a $2$-vertex (see Figure~\ref{fig:3_03_13_1-2}). Having colored the rest of $(G, \sigma)$ except for the $2$-neighbor of $z$, we are left with a list assignment on $u$, $x$, $y$, $z$, $w$ where each of $L(x)$ and $L(y)$ is a neighbored $5$-set, $L(w)$ is a full list, and by applying Lemma~\ref{lem:3-1} to the $uz$-branch of the figure, we modify the list of $u$ to a paired $8$-set. We may then apply Lemma~\ref{lem:evencycle}~\ref{even3} on the $4$-cycle $uxwy$ and we are done.

   Finally we show that $w$ cannot be a $3_0$-vertex either. Depending on whether $w$ is adjacent to $z$ or not, we have two cases to consider: 

\medskip
\noindent
{\bfseries case 1:} $w\sim z$. We follow the labeling of Figure~\ref{fig:3_0neighborhood2}. 

	\begin{figure}[h]
		\centering
		\begin{tikzpicture}[>=latex,
		roundnode/.style={circle, draw=black!90, very thick, minimum size=5mm, inner sep=0pt},
		roundnode2/.style={circle, draw=black!90, very thick, minimum size=5mm, inner sep=0pt},
		squarednode/.style={rectangle, draw=red!60, very thick, minimum size=5mm, inner sep=0pt}, scale=0.88
		] 
		\node [roundnode2] (u) at (3,3){$u$};
		\node [roundnode2] (y) at (3,2){$y$};
		\node [roundnode2](w) at (3,1){$w$};
		\node [roundnode2] (x) at (0,2){$x$};
		\node [roundnode] (x1) at (-0.8,2){$x_1$};
		\node [squarednode] (x2) at (-1.6,2){};
        \node [roundnode] (y1) at (3.8,2){$y_1$};
        \node [squarednode] (y2) at (4.6,2){};
		\node [roundnode2] (z) at (6,2){$z$};
		\node [roundnode] (z1) at (6.8,2){$z_1$};
		\node [squarednode] (z2) at (7.6,2){};
		
		\draw [line width =1.4pt, gray] (u)--(x)--(w);
		\draw [line width =1.4pt, gray] (x)--(x1)--(x2);
		\draw [line width =1.4pt, gray] (y)--(y1)--(y2);
		\draw [line width =1.4pt, gray] (z)--(z1)--(z2);
		\draw [line width =1.4pt, gray] (u)--(y)--(w);	
		\draw [line width =1.4pt, gray] (u)--(z)--(w); 
		\end{tikzpicture}
		\caption{$w$ is adjacent to $z$}
		\label{fig:3_0neighborhood2}
	\end{figure}

Assuming that rest of the graph is precolored, to extend the coloring to this part of the graph we have a full list of colors available on $u$ and $w$, and on each of $x$, $y$ and $z$ we have a paired $10$-set available, or, equivalently, only a paired $2$-set is missing. In what follows, we will try three possible partial coloring $\phi$ of $u$ and $w$, for each choice we either can extend $\phi$ to the full configuration, or will give a condition on the lists for $x$, $y$ and $z$. Then we will find fourth assignment to $u$ and $w$ that is extendable.

Our first coloring to consider satisfies that $\phi(u)=1^+$ and $\phi(w)=3^+$. This coloring can be extended to $x$, $y$ and $z$ unless for one of them, say $x$, one of the following holds: (1) both $ux$ and $xw$ are negative edges and the missing pair on $x$ is $\{5^-, 6^-\}$ or (2) both $ux$ and $xw$ are positive edges and the missing pair on $x$ is $\{5^+, 6^+\}$. 

As second choice, we try the coloring $\phi(u)=1^+$ and $\phi(w)=5^+$. Similarly, if this choice of colors is not extendable, for a vertex, say $y$, either (3) both $uy$ and $yw$ are negative edges and the missing two colors are $\{3^-, 4^-\}$ or (4) both $uy$ and $yw$ are positive edges and the missing pair on $y$ is $\{3^+, 4^+\}$, which, in particular, justifies the choice $y\neq x$.

As a third try, on examining the coloring $\phi(u)=1^+$ and $\phi(w)=3^-$, we conclude that for one of the three vertices, say $z$, either (5) $uz$ is positive and $zw$ is negative with $\{ 5^+, 6^+\}$ as the missing pair on $z$ or (6) $uz$ is negative and $zw$ is positive with $\{ 5^-, 6^-\}$ as the missing pair on $z$. Theses conditions also justify that $z$ is distinct from both $x$ and $y$.  

We now observe that the choice of $\phi(u)=1^+$ and $\phi(w)=5^-$ is extendable on all three of $x$, $y$, $z$.
		
\medskip
\noindent
{\bfseries case 2:} $w\nsim z$. We will use the labeling of Figure~\ref{fig:3_0neighborhood3}.

	\begin{figure}[h]
	\centering
	\begin{tikzpicture}[>=latex,
	roundnode/.style={circle, draw=black!90, very thick, minimum size=5mm, inner sep=0pt},
	roundnode2/.style={circle, draw=black!90, very thick, minimum size=5mm, inner sep=0pt},
	squarednode/.style={rectangle, draw=red!60, very thick, minimum size=5mm, inner sep=0pt}, scale=0.88
	] 
	\node [roundnode2] (u) at (3,3){$u$};
	\node [roundnode2] (y) at (3,2){$y$};
	\node [roundnode2](w) at (1.5,1){$w$};
	\node [roundnode2] (x) at (0,2){$x$};
	\node [roundnode] (x1) at (-0.8,2){$x_1$};
	\node [squarednode] (x2) at (-1.6,2){};
	\node [roundnode] (y1) at (3.8,2){$y_1$};
	\node [squarednode] (y2) at (4.6,2){};
	\node [roundnode2] (z) at (6,2){$z$};
	\node [roundnode] (z1) at (6.8,2){$z_1$};
	\node [squarednode] (z2) at (7.6,2){};
	\node [squarednode] (w1) at (1.5,0){};
	\node [squarednode] (z3) at (6,1){};
	
	\draw [line width =1.4pt, gray] (u)--(x)--(w);
	\draw [line width =1.4pt, gray] (x)--(x1)--(x2);
	\draw [line width =1.4pt, gray] (y)--(y1)--(y2);
	\draw [line width =1.4pt, gray] (z)--(z1)--(z2);
	\draw [line width =1.4pt, gray] (u)--(y)--(w)--(w1);	
	\draw [line width =1.4pt, gray] (u)--(z)--(z3); 
	\end{tikzpicture}
	\caption{$w$ is adjacent to $z$}
	\label{fig:3_0neighborhood3}
\end{figure}

Upon forming list of available colors on $u$, $x$, $w$ and $y$ using $zz_1$-branch for $u$, $x_1$-branch for $x$, and $y_1$-branch for $y$, the list of $u$, by Lemma~\ref{lem:3-1}, is a paired $8$-set. The list of $w$ is a neighbored $5$-set and the list of each of $x$ and $y$ is a paired $10$-set, equivalently, only a paired 2-set of colors is missing at $x$ or $y$. In particular, there is one color from each layer available at $w$. For one such color, say $c$, there must be three pairs of colors available for $u$ each not in the same layer as $c$. Let $c_1, c_2$ and $c_3$ each be a color from one of these pairs. We may now proceed as in the previous lemma, assigning $c$ to $w$ and $c_i$ to $u$ would be not extendable only if $x$ or $y$ is of a certain type, but there are only two of these vertices and three distinct possibilities.
\end{proof}

\subsubsection{Paths in 3-subgraph}\label{sec:badchain}
We have so far seen that $(G, \sigma)$, the minimum counterexample to our claim, have no $3_2$-vertex and no $3_1$-vertex seeing two other $3_1$-vertices. To complete our proof, we need further information on the subgraph induced by $3$-vertices. When applying discharging technique in the next section, among $3_0$-vertices the poorest one would be: $1.$ a $3_0$-vertex all whose neighbors are $3_1$-vertices, $2.$ a $3_0$-vertex with two $3_1$-neighbors one of which has another $3_1$-neighbor. A path of $(G, \sigma)$ is said to be \emph{poor} if first of all its vertices are alternatively of type $3_0$ and $3_1$, and secondly, the first and the last vertices of the path are among the poorest type of $3_0$-vertices. 

Our goal is to show that $(G, \sigma)$ does not contain a poor path, to this end, we will assume that $P$ is a minimum poor path in $(G, \sigma)$ whose vertices are labeled $v_1v_2\cdots v_{2k+1}$. If the end vertex $v_1$ is of type 1, then we label its other $3_1$-neighbors $v_0$ and $v'_0$ and if it is of type 2, then its other $3_1$-neighbor is labeled $v_0$ and the other $3_1$-neighbor of $v_0$ is labeled $v_{-1}$. Vertices $v_{2k+2}$, $v'_{2k+2}$ and $v_{2k+3}$ are defined similarly. Observe that, by Lemma~\ref{lem:RC3_1-3_1-3_1}, vertices $v_{-1}$ and $v_{2k+3}$, when exist, are two distinct vertices. Thus for $k\geq 1$, depending on the types of two ends of the poor path, we have three possible types of poor paths. For $k=0$, $v_1$ is viewed as the end vertex from each direction, but as it is a $3$-vertex, it cannot be of type 1 from each end, thus we can only have two types of poor paths. In Figure~\ref{fig:PoorPath}, both of these two possibilities are depicted, where only one possibility for $k\geq 1$ is also presented.

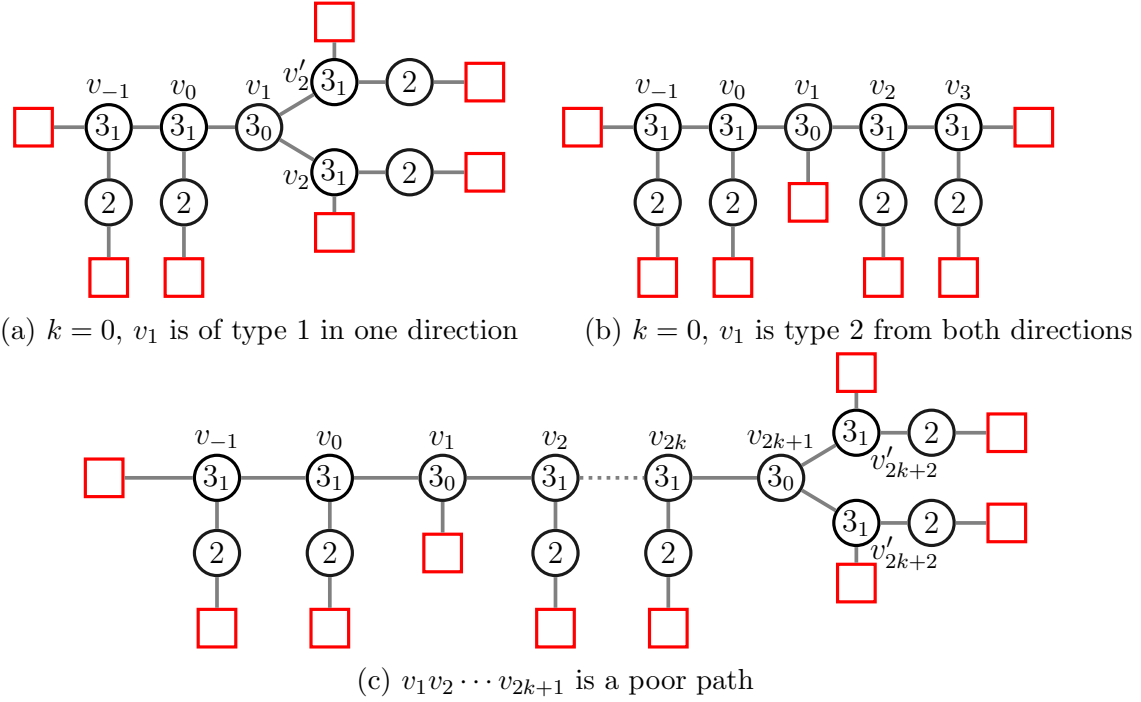
\begin{figure}[htbp]
\begin{minipage}[t]{0.5\textwidth}
\centering
\begin{tikzpicture}[>=latex,
roundnode/.style={circle, draw=black!90, very thick, minimum size=6mm, inner sep=0pt},
blacknode/.style={circle, draw=black, very thick, minimum size=6mm, inner sep=0pt},
bluenode/.style={circle, draw=black!90, very thick, minimum size=6mm, inner sep=0pt},
squarednode/.style={rectangle, draw=red, very thick, minimum size=5mm, inner sep=0pt}, scale=1
] 
\node [blacknode] (x) at (2,1){$3_1$};
\node at (2,1.5){$v_{-1}$};
\node [bluenode] (x1) at (2,0){$2$};
\node [squarednode] (x0) at (1,1){};
\node [squarednode] (x2) at (2,-1){};
\node [blacknode] (y) at (3,1){$3_1$};
\node at (3,1.5){$v_0$};
\node [bluenode] (y1) at (3,0){$2$};
\node [squarednode] (y2) at (3,-1){};
\node [roundnode] (v) at (4,1){$3_0$};
\node at (4,1.5){$v_1$}; 

\node [blacknode](x'_{2k+2}) at (5,1.6){$3_1$};
\node   at (4.5,1.7) {$v'_{2}$};
\node [blacknode](x_{2k+2}) at (5,0.4){$3_1$};
\node   at (4.5,0.3){$v_{2}$};
\node [squarednode](p) at (5,2.4){};
\node [bluenode](p1) at (6,1.6){$2$};
\node [squarednode](p2) at (7,1.6){};

\node [bluenode](q1) at (6,0.4){$2$};
\node [squarednode](q) at (5,-0.4){};
\node [squarednode](q2) at (7,0.4){};

\draw [line width =1.4pt,-, gray] (x)--(y)--(v);
\draw [line width =1.4pt, gray] (v)--(x'_{2k+2});
\draw [line width =1.4pt, gray] (v)--(x_{2k+2});
\draw [line width =1.4pt, gray, -] (x0)--(x)--(x1)--(x2);
\draw [line width =1.4pt, gray, -] (y)--(y1)--(y2); 
\draw [line width =1.4pt, gray, -] (p)--(x'_{2k+2})--(p1)--(p2);	
\draw [line width =1.4pt, gray, -] (q)--(x_{2k+2})--(q1)--(q2); 
\end{tikzpicture}
\subcaption{$k=0$, $v_1$ is of type 1 in one direction}
\end{minipage} 
\begin{minipage}[t]{0.5\textwidth}
\begin{tikzpicture}[>=latex,
roundnode/.style={circle, draw=black!90, very thick, minimum size=6mm, inner sep=0pt},
blacknode/.style={circle, draw=black, very thick, minimum size=6mm, inner sep=0pt},
bluenode/.style={circle, draw=black!90, very thick, minimum size=6mm, inner sep=0pt},
squarednode/.style={rectangle, draw=red, very thick, minimum size=5mm, inner sep=0pt}, scale=1
] 
\node [blacknode] (x) at (2,1){$3_1$};
\node at (2,1.5){$v_{-1}$};
\node [bluenode] (x1) at (2,0){$2$};
\node [squarednode] (x0) at (1,1){};
\node [squarednode] (x2) at (2,-1){};
\node [blacknode] (y) at (3,1){$3_1$};
\node at (3,1.5){$v_0$};
\node [bluenode] (y1) at (3,0){$2$};
\node [squarednode] (y2) at (3,-1){};
\node [roundnode] (v) at (4,1){$3_0$};
\node at (4,1.5){$v_1$};
\node [squarednode] (v1) at (4,0){};
\node [blacknode](w) at (5,1){$3_1$};
\node at (5,1.5){$v_2$};
\node [blacknode](z) at (6,1){$3_1$}; 
\node at (6,1.5){$v_3$};
\node [bluenode](w1) at (5,0){$2$};
\node [squarednode](w2) at (5,-1){};	
\node [bluenode](z1) at (6,0){$2$};
\node [squarednode](z2) at (6,-1){};
\node [squarednode](z3) at (7,1){};
\draw [line width =1.4pt, gray, -] (x)--(y)--(v)--(w);
\draw [line width =1.4pt, gray] (w)--(z);
\draw [line width =1.4pt, gray, -] (x0)--(x)--(x1)--(x2);
\draw [line width =1.4pt, gray, -] (y)--(y1)--(y2);
\draw [line width =1.4pt, gray, -] (w)--(w1)--(w2);
\draw [line width =1.4pt, gray, -] (z3)--(z)--(z1)--(z2);	
\draw [line width =1.4pt, gray, -] (v)--(v1);	        
\end{tikzpicture}
\subcaption{$k=0$, $v_1$ is type 2 from both directions}
\end{minipage} 

\begin{minipage}[t]{\textwidth}
\centering
\begin{tikzpicture}[>=latex,
roundnode/.style={circle, draw=black!90, very thick, minimum size=6mm, inner sep=0pt},
blacknode/.style={circle, draw=black, very thick, minimum size=6mm, inner sep=0pt},
bluenode/.style={circle, draw=black!90, very thick, minimum size=6mm, inner sep=0pt},
squarednode/.style={rectangle, draw=red, very thick, minimum size=5mm, inner sep=0pt}, scale=1
] 
\node [blacknode] (x) at (2,1){$3_1$};
\node at (2,1.5){$v_{-1}$};
\node [bluenode] (x1) at (2,0){$2$};
\node [squarednode] (x0) at (0.5,1){};
\node [squarednode] (x2) at (2,-1){};
\node [blacknode] (y) at (3.5,1){$3_1$};
\node at (3.5,1.5){$v_0$};
\node [bluenode] (y1) at (3.5,0){$2$};
\node [squarednode] (y2) at (3.5,-1){};
\node [roundnode] (v) at (5,1){$3_0$};
\node  at (5,1.5){$v_1$};
\node [squarednode] (v1) at (5,0){};
\node [roundnode](w) at (6.5,1){$3_1$};
\node  at (6.5,1.5){$v_{2}$};
\node [roundnode](z) at (8,1){$3_1$};
\node  at (8,1.5){$v_{2k}$};
\node [roundnode](x_{2k+1}) at (9.5,1){$3_0$};
\node  at (9.5,1.5){$v_{2k+1}$};
\node [blacknode](x'_{2k+2}) at (10.5,1.6){$3_1$};
\node at (11.14,1.2){$v'_{2k+2}$};
\node [blacknode](x_{2k+2}) at (10.5,0.4){$3_1$};
\node at (11.14,0){$v'_{2k+2}$};

\node [bluenode](w1) at (6.5,0){$2$};
\node [squarednode](w2) at (6.5,-1){};

\node [bluenode](z1) at (8,0){$2$};
\node [squarednode](z2) at (8,-1){};

\node [squarednode](p) at (10.5,2.4){};
\node [bluenode](p1) at (11.5,1.6){$2$};
\node [squarednode](p2) at (12.5,1.6){};

\node [bluenode](q1) at (11.5,0.4){$2$};
\node [squarednode](q) at (10.5,-0.4){};
\node [squarednode](q2) at (12.5,0.4){};

\draw [line width =1.4pt, gray, -] (x)--(y)--(v)--(w);
\draw [line width =1.4pt, gray, dotted] (w)--(z);
\draw [line width =1.4pt, gray] (x_{2k+1})--(x'_{2k+2});
\draw [line width =1.4pt, gray] (x_{2k+1})--(x_{2k+2});
\draw [line width =1.4pt, gray, -] (x0)--(x)--(x1)--(x2);
\draw [line width =1.4pt, gray, -] (y)--(y1)--(y2);
\draw [line width =1.4pt, gray, -] (w)--(w1)--(w2);
\draw [line width =1.4pt, gray, -] (z)--(z1)--(z2);	
\draw [line width =1.4pt, gray, -] (v)--(v1);
\draw [line width =1.4pt, gray, -] (p)--(x'_{2k+2})--(p1)--(p2);	
\draw [line width =1.4pt, gray, -] (q)--(x_{2k+2})--(q1)--(q2);	
\draw  [line width=1.4pt, gray, -] (z)--(x_{2k+1});  		        
\end{tikzpicture}
\subcaption{$v_1v_2\cdots v_{2k+1}$ is a poor path}
\end{minipage} 
\caption{ Poor paths}
\label{fig:PoorPath}
\end{figure}

Let $I_{3_0} = \{1,3,\ldots, 2k-1,2k+1\}$ and $I_{3_1} = \{0,2,\ldots, 2k, 2k+2\}$. As every $3_1$-vertex has a $2$-neighbor, for $i \in I_{3_1}\setminus\{0, 2k+2\}$, the vertex $v_i$ is not adjacent to any vertex $v_j$ for $j \in I_{3_0} \cup I_{3_1} \setminus\{i-1,i+1\}$. Moreover, by Lemma \ref{lem:RC3_0-three-3_1-etc}, vertices $v_0$ and $v_0'$, when the latter exists, do not share a $2$-neighbor. Furthermore, vertices $v_0$ and $v_0'$ are not distinguishable in a poor path so we can switch their roles if necessary. We can similarly treat $v_{2k+2}$ and $v'_{2k+2}$.

\begin{lemma}\label{cla1}
In the minimum counterexample $(G, \sigma)$ and with $P$ as a minimum poor path, the vertex $v_0$ is not adjacent to $v_i$ for $i \in  \{2,3,\ldots, 2k, 2k+1\}$, and the vertex $v_{2k+2}$ is not adjacent to $v_j$ for $j \in \{1,2,\ldots, 2k\}$. Moreover, $v_0 \neq v_{2k+2}$.
\end{lemma}

\begin{proof}
We give a proof for $v_0$, the argument for $v_{2k+2}$ follows by symmetries. It is already mentioned in the paragraph proceeding the lemma that $v_0$ is not adjacent to $v_i$ for even values of $i$. Thus, we only need to consider the odd values of $i$.

In the case of $k=0$, because $v_1$ is of degree $3$, at least one side is of type 2, thus one of $v_0$ or $v_2$ has its degree already full and both of the claims follow.

Next we consider the case $k=1$. The claim is that $v_0$ is not adjacent to $v_3$. By contradiction suppose it is. Since $v_0$ is a $3_1$-vertex, we already have $v_1$ and $v_3$ as the neighbors of $v_0$ which are not $2$-vertices. This, in particular, implies that $v_3$ is of type 1 and that $v_0=v_4$, and $v'_4$ is the other $3_1$-neighbor of $v_3$. We may now apply Lemma~\ref{lem:RC3_0-three-3_1-etc} with $u=v_3$, this completes the proof for $k=1$. 

For $k\geq 2$, and for the first part of the claim, observe that $v_1$ must be of type 1, as otherwise $v_1$ and $v_{-1}$ are the only $3$-neighbors of $v_0$.  Assume to the contrary that $v_0$ is adjacent to $v_j$ for an odd value of $j$. We now get a contradiction with the minimality of $P$ by considering the shorter poor path: $v_jv_{j+1}\cdots v_{2k+1}$.
It remains to show that $v_0\neq v_{2k+2}$. Again, suppose to the contrary that $v_0= v_{2k+2}$. Thus $v_0$, which is a $3_1$-vertex, is adjacent to both $v_1$ and $v_{2k+1}$, and, therefore, it has no other $3$-neighbor. Hence both $v_1$ and $v_{2k+1}$ are of type 1. But again we get a contradiction to the minimality of $P$ by taking the shorter poor path: $v_1v_0v_{2k+1}$.
\end{proof}

\begin{lemma}\label{cla2}
In the minimum counterexample $(G, \sigma)$ and with $P=v_1v_2\cdots v_{2k+1}$ as a minimum poor path, the following statements hold:
	\begin{enumerate}[leftmargin =3em, label=(\alph*)]
		\item  For any $i \in I_{3_1} \cup \{-1\} $ and $j \in I_{3_1} \cup \{2k+3\}$, the vertices $v_i$ and $v_j$ do not have a common $2$-neighbor.
		\item  For any $i \in I_{3_0} \cup \{-1\} $, $j \in I_{3_0} \cup \{2k+3\}$, the vertex $v_i$ is not adjacent to the vertex $v_j$.
         \item The vertex $v_0$ is not adjacent to $v_{2k+2}$.
	\end{enumerate}
\end{lemma}

\begin{proof}
        We prove the first two claims by contradiction. We consider all the pairs $i, j$ for which one of the two statements does not hold. Among all such pairs then we choose one where $j$ is the minimum possible and, based on this condition, $i$ is the maximum possible. Then, depending on which of the statement fails for this pair of $i,j$, we consider two separate cases.

		\medskip
		\noindent 
		{\bfseries Case 1.} The statement $(a)$ does not hold for $i$ and $j$.  We will consider four subcases based on $i$ and $j$. 
 
\begin{itemize}
\item $(i,j)=(-1,2k+3)$. Thus, in particular, we assume $v_{-1}$ and $v_{2k+3}$ exists and that they are distinct from other vertices, hence, $v_0$ is not adjacent to $v_{2k+2}$. Let $(H,\sigma)$ be the subgraph of $(G, \sigma)$ induced by the vertices of $P$, the vertices $v_{-1},v_0,v_{2k+2}, v_{2k+3}$ and all their $2$-neighbors. Let $u$ be the common 2-neighbor of $v_{-1}$ and $v_{2k+3}$. Observe that, by the maximality of $j$ and the minimality of $i$, expect for $u$, every other $2$-vertex in $(H, \sigma)$ sees only one vertex in $(H,\sigma)$ and that there is no connection between $3_0$-vertices of $(H,\sigma)$. We may then color $(G,\sigma)-(H, \sigma)$ by the minimality of $(G,\sigma)$, and with respect to this partial coloring, consider the list of available colors on the vertices of $(H, \sigma)$. Observe that if we remove all 2-vertices but $u$ from $(H, \sigma)$, we have a subgraph $(H',\sigma)$ which is a signed $(2k+6)$-cycle. Furthermore, by Lemma~\ref{lem:2-Path}, the list coloring problem on $(H,\sigma)$ can be modified to a list coloring problem on $(H', \sigma)$ where $L'(u)=C$, $L'(v_{-1})$ and $L'(v_{2k+3})$ each is a neighbored $5$-set, each $L'(v_i)$, $i=0, 2, \ldots, 2k+2$, is a paired $10$-set and each $L'(v_j)$, $j=1, 3, \ldots, 2k+1$, is a neighbored $5$-set. But then, by Lemma~\ref{lem:evencycle}~\ref{even1}, we do have a coloring of $(H', \sigma)$ with respect to this list assignment $L'$.

\item $i=-1$, $j\in I_{3_1}$. Let $u$ be the common neighbor of $v_{-1}$ and $v_{j}$. Similar to the previous case, we consider the subgraph $(H,\sigma)$ induced by $v_{-1}, v_{0}, \ldots, v_{j}$ and all their 2-neighbors, noting that, by the choice of $j$ and $i$, each such a $2$-neighbor is adjacent to only one vertex in $(H, \sigma)$ and that no two $3_0$-vertices in $(H, \pi)$ are adjacent. Thus the subgraph $(H', \sigma)$ induced by $u$ and 3-vertices of $(H, \sigma)$ is a $(j+3)$-cycle.
Again, similar to the previous case, a coloring $\phi$ of $(G, \sigma)-(H,\sigma)$ induces a list assignment on $(H',\sigma)$ which satisfies the conditions of Lemma~\ref{lem:oddcycle}~\ref{odd2}, therefore, $\phi$ can be extend to the rest of $(G, \sigma)$.

\item $i=0$. By symmetry of $1$ and $2k+1$, we may assume $j\in I_{3_1}$. First we note that if $j=2k+2$, then $v_0$ is not adjacent to $v_{2k+2}$ as otherwise we have the forbidden configuration of Lemma~\ref{lem:RC3_1-2-3_1}. Let $u$ the common 2-neighbor of $v_i$ and $v_j$. For this case we will consider two subcases based on the type of the vertex $v_1$. If $v_1$ is of type 1, then we take $(H, \sigma)$ to be the subgraph induced by $v_0, v_1, v_2, \ldots, v_j$, $v'_0$ and all their 2-neighbors. Let $\phi$ be a coloring of $(G,\sigma)-(H, \sigma)$. Let $L$ be the list assignment induced on $(H, \sigma)$ by the partial coloring $\phi$. This $L$-coloring problem is reduced to an $L'$-coloring problem of the cycle $v_0v_1\cdots v_ju$ where $L'(u)=C$, $L'(v_0)$ and $L'(v_j)$ are neighbored 5-sets, $L'(v_1)$, by Lemma~\ref{lem:3-1}, contains at least one paired 4-set from $C^+$ and one paired 4-set from $C^-$, and the rest of $L'(v_k)$ are alternatively neighbored $5$-sets and paired $10$-sets. 
Overall this cycle with respect to $L'$ satisfies the conditions of Lemma~\ref{lem:evencycle}~\ref{even3}, and, therefore, the coloring $\phi$ can be extended to the rest of $(G,\sigma)$.

If $v_1$ is of type 2, then, by similar argument, the problem is reduced to the $L'$-coloring of the cycle $v_0v_1\cdots v_ju$ where the lists of $v_0$ and $v_1$ have changed the roles, with all other remaining the same as before. We may then apply Lemma~\ref{lem:evencycle}~\ref{even2} to complete the proof.

\item $i\in I_{3_1}, i\geq 2$. By the symmetry of $i=0$ and $j=2k+2$ we may assume $j\in I_{3_1}$, $j\neq 2k+2$. As in the previous cases, we let $(H,\sigma)$ be the subgraph induced by $v_0,v_1, \ldots v_j$, one of $v_{-1}$ or $v'_{0}$ depending on the type of $v_1$, and all the $2$-neighbors of already chosen vertices. Let the common 2-neighbor of $v_i$ and $v_j$ be $u$ and note that all other 2-neighbors of the vertices in $(H, \sigma)$ are distinct. Furthermore, no pair of $3_0$-vertices in $(H,\sigma)$ are adjacent. Assume that $(G,\sigma)-(H,\sigma)$ admits a list-coloring $\phi$ and let $L$ be the associated list assignment on $(H, \sigma)$. As before, we will reduce the $L$-coloring problem of $(H, \sigma)$ to an $L'$-coloring of the cycle $uv_iv_{i+1}\cdots v_j$ which satisfies the conditions of Lemma~\ref{lem:evencycle}~\ref{even2}. To get $L'$, if $v_1$ is of type 1 we apply Lemma~\ref{lem:3-1} to $v_1$ from two directions after which we have a paired 4-set of colors available at $v_1$. Then using Lemma~\ref{oneEdgeRestriction} \ref{P2-1.5}, \ref{P2-2}, \ref{P2-3} and Lemma~\ref{lem:2-Path}, we update the lists of vertices $v_l$ of $P$ with $l\leq i$ such that we have, alternatively, lists of size $6$ and $3$ until $v_{i-1}$, and $|L'(v_i)|\geq 8$. The case when $v_1$ is of type 2 is quite similar. 
The only difference is that at the start $v_1$ would a neighbored $3$-set rather than a paired $4$-set.
\end{itemize} 

{\bfseries Case 2.} The statement $(b)$ does not hold for $i$ and $j$. The proof technique is quite similar to the previous case with less subcases to consider, so we only give the general idea. The case of $i=-1, j=2k+3$ is not possible by Lemma~\ref{lem:RC3_1-3_1-3_1}. In all other subcases, we consider the subgraph $(H, \sigma)$ induced by vertices $v_0,v_1,v_2, \ldots, v_j$, one of $v_{-1}$ or $v'_{0}$, and their $2$-neighbors. The problem then is reduced to a list coloring problem on $(H, \sigma)$, but as $(H, \sigma)$ has a unique cycle, we may further reduce the problem to list coloring of the cycles. However, in all but one of the cases we may apply Lemma~\ref{lem:oddcycle}~\ref{odd1}. In the exceptional case when  $v_0$ is adjacent to $v_{2k+2}$ and $v_1$ is adjacent to $v_{2k+1}$, we consider the $4$-cycle $v_0v_{2k+2}v_{2k+1}v_1$ and let $(H, \sigma)$ be the subgraph induced by this $4$-cycle and the two $2$-neighbors of $v_0$ and $v_{2k+2}$. Then a coloring of $(G,\sigma)-(H,\sigma)$ can be extended to $(H, \sigma)$ by Lemma~\ref{lem:List-4-Cycle}.

Finally, we prove part $c$: $v_0$ is not adjacent to $v_{2k+2}$. Assume to the contrary, $v_0$ is adjacent to $v_{2k+2}$. Then $v_0 v_1 \cdots v_{2k+2}$ is a cycle. In the above arguments, we have shown that first of all there is no chord in this cycle, secondly, for any two $3_1$-vertices of the cycle, their $2$-neighbors are distinct. As before we consider the signed subgraph $(H, \sigma)$ induced by the cycle and its 2-neighbors. Then again a coloring of  $(G,\sigma)-(H,\sigma)$ can be extended to $(H, \sigma)$ by Lemma~\ref{lem:oddcycle}~\ref{odd1}.
\end{proof}

\begin{lemma}\label{lem:RC30-31-chain}
	The minimum counterexample $(G,\sigma)$ contains no poor path.
\end{lemma}
\begin{proof}
	Assume to the contrary and let $P=v_1\cdots v_{2k+1}$ be a minimum poor path. Let $(H,\sigma)$ be the subgraph of $(G,\sigma)$ induced by $v_0,v_1,\ldots, v_{2k+2}$, $v'_0$ or $v_{-1}$ depending on the type of $v_0$, $v'_{2k+2}$ or $v_{2k+3}$ depending on the type of $v_{2k+1}$ and all of their $2$-neighbors. Observe that by Lemma~\ref{cla1} and Lemma~\ref{cla2}, $H$ is a tree. By the minimality of $(G, \sigma)$, we have a coloring of $(G, \sigma)-(H,\sigma)$ which induces a list assignment $L$ on the signed tree $(H,\sigma)$. To complete the proof, having considered $v_1$ as the root of this tree we will show that $L^a(v_1)\neq \emptyset$.

    If $k=0$, then $H$ is either the graph $(a)$ of Figure~\ref{fig:PoorPath} or the graph $(b)$. In either case, to compute the number of colors forbidden on $v_{0}$, we apply Lemma~\ref{lem:3-1} to the $v_{-1}$-branch and Lemma~\ref{lem:2-Path} to the $2$-vertex branch, concluding that 
    $L^a(v_{0})$ contains a paired 6-set. Thus, by Lemma~\ref{oneEdgeRestriction}~\ref{P2-3}, the $v_0$-branch will forbid at most a pair of colors from $L(v_1)$. 
    If we have the case $(a)$ of the figure, then each of $v_2$ and $v'_2$, by Lemma~\ref{lem:3-1}, will forbid at most two pairs from $L(v_1)$, altogether we have at most five pairs, thus in all case $L^a(v_{1})$ contains at least one pair of color. 
    If we have the case $(b)$ of the figure, then by symmetry of $v_0$ and $v_2$ we have at most a pair forbidden from $L(v_1)$ by $v_2$. In this case $L(v_1)$ was a neighbored 5-set, this $L^a(v_{1})$ still contains at least one element.

    For $k\geq 1$, depending on the type of $v_{2k+1}$ and just as in the previous case, $L^a(v_{2k+1})$ contains either a paired 4-set or a neighbored 3-set. Then, by Lemma~\ref{lem:2-Path} and Lemma~\ref{oneEdgeRestriction} (\ref{P2-2} or \ref{P2-1.5}), $L^a(v_{2k})$ contains a paired $6$-set, which in turn implies that $L^a(v_{2k-1})$ is a neighbored $3$-set. Repeating this process, $L^a(v_{2})$ contains a paired $6$-set, thus from this branch of the tree at most one pair of colors will be forbidden on $v_1$. Now if $v_1$ is of type 1, then the branches corresponding to $v_0$ and $v'_0$ each may forbid at most two pairs of colors, and since $L(v_1)=C$, we still have a pair of available colors. If $v_1$ is of type 2, then the $v_0$-branch forbids only one pair, and since $L(v_1)$ is a neighbored 5-set, we still have a color available at $v_1$.
\end{proof}

\subsection{Discharging Method}\label{sec:discharging}

Recall that $(G, \sigma)$ is a minimum counterexample to Theorem~\ref{main-theorem}. A \emph{$3^-$-subgraph} of $G$ is a connected component $H$ of the subgraph induced by the set of 3-vertices and 2-vertices of $G$. Given a $3^-$-subgraph $H$, let $n_0(H)$ be the number of vertices of $H$ which are $3_0$-vertices in $G$ and let $n_1(H)$ be the number of $3_1$-vertices of $G$ that are of degree 3 in $H$, i.e., the number of vertices in $H$ each of which has one 2-neighbor and two 3-neighbors. 

\begin{lemma}\label{lem:n_0-n_1}
	In any $3^-$-subgraph $H$ of $G$, $n_0(H) \geq n_1(H)$.
\end{lemma}

\begin{proof}
Our proof is by discharging technique. We assign an initial charge of 1 to all in $H$ that are $3_0$-vertices of $G$ and a charge of 0 to all other vertices of $H$. We will introduce discharging rules and prove that, upon applying these rules, each vertex in $H$ which is a $3_1$-vertex of $G$ receives a total charge of $\frac{1}{2}$ while no $3_0$-vertex of $G$ in $H$ loses more than $\frac{1}{2}$. That would prove our claim. The discharging rules we use are as follows.
	
{\bf Rule 1} Given a $3_0$-vertex $v_1$ of $G$, assume there exists a unique path $P=v_1\cdots v_{2k+1}$, $k \geq 0$, satisfying the followings:  (1) For every odd value of $i$, the vertex $v_i$ is a $3_0$-vertex of $G$ and for every even value of $i$, the vertex $v_i$ is a $3_1$-vertex of $G$. (2) Either $i.$ $v_{2k+1}$ has two other neighbors that are $3_1$-vertices of $G$, or $ii.$ $v_{2k+1}$ has one $3_1$-neighbor, $v_{2k+2}$, which itself has a $3_1$-neighbor in $G$ ($P$ can be seen as one end of a poor path).  Then in case $i.$ if $k\geq 1$, $v_1$ gives a charge of $\frac{1}{2}$ to $v_2$, in case $ii.$ $v_1$ will give a charge of  $\frac{1}{2}$ to $v_2$ (that is even if $k=0$).
  
{\bf Rule 2} Each $3_1$-vertex of $G$ which is of degree 3 in $H$ and is of charge 0 after Rule 1, receives a charge of $\frac{1}{4}$ from each of its $3_0$-neighbor.

First observe that a $3_1$-vertex $x$ of $G$ which is of degree 3 in $H$, by Lemma~\ref{lem:RC3_1-3_1-3_1}, has at least one $3_0$-neighbor say $y$. If the other $3$-neighbor $z$ of $x$ is a $3_1$-vertex, then $P=y$ is a path  described in Rule 1 of type $ii.$ where $k=0$ and $x=v_{2}$, moreover, this is unique such a path as any other such a path $P'$ together with $P$ will form a poor path, contradicting Lemma~\ref{lem:RC30-31-chain}. Therefore, by Rule 1, $x$ will receive a charge of $\frac{1}{2}$ from $y$. If $z$ is also a $3_0$-vertex then either it receives a charge of $\frac{1}{2}$ from one of $y$ or $z$ when applying Rule 1 or, it will receive a charge of $\frac{1}{4}$ from each of them, thus in all cases it will have a final charge of $\frac{1}{2}$.

It remains to show that no $3_0$-vertex of $G$ in $H$ will lose more than $\frac{1}{2}$ of its charge. Since $G$ has no poor path, and that Rule 1 can only apply if there is a unique path $P$, it may only apply in one direction on a given $3_0$-vertex. Thus Rule 1, on its own, will take a charge of at most $\frac{1}{2}$ from a $3_0$-vertex. 

Next we consider a $3_0$-vertex $u$ which has three $3_1$-neighbors $u_1$, $u_2$ and $u_3$ each of which is a $3$-vertex of $H$. Let $u'_1$, $u'_2$ and $u'_3$ be neighbors of $u_1$, $u_2$ $u_3$, respectively, which are not $u$ and not $2$-vertices. Thus each of them has to be a $3_0$-vertex of $G$ as otherwise we have a poor path with $k=0$.  First we assume that $u'_1u_1uu_2u'_2$ is a part of a cycle where vertices are alternatively $3_0$-vertices and $3_1$-vertices of $G$. We claim that in this case $u'_3u_3u$ is the unique path $P$ satisfying Rule 1. Otherwise, a second path $P'$ starting at $u'_3$ exists. If $P'$ has no common vertex with $P$, then $P$ and $P'$ together form a poor path, contradicting Lemma~\ref{lem:RC30-31-chain}. Else $P'$ must intersect the cycle to reach $u$, in which case the common part of $P'$ and the cycle form a poor path. Thus $u_3$ receives a charge of $\frac{1}{2}$ from $u'_3$ by Rule 1. When applying Rule 2, $u$ loses only a total charge of $\frac{1}{2}$. When there is no such a cycle, then each of $u'_iu_iu$, $i=1,2,3$, is a path satisfying Rule 1 with $k=1$, furthermore, each of them satisfies the condition of being unique, as otherwise we will have a poor path. Thus $u$ will lose no charge in this case.

It remains to show that if a $3_0$-vertex $u$ has given a charge of $\frac{1}{2}$ to a $3_1$-neighbor $u_1$ by Rule 1 then $u$ will not lose any charge by Rule 2. Let $u_2$ be another $3_1$-neighbor of $u$ which is a $3$-vertex of $H$. Let $u'_2$ be the other neighbor of $u_2$ which is not a $2$-vertex. We claim that $u'_2$ is a $3_0$-vertex. Otherwise, together with the path $P$ (of Rule 1) we have a poor path. Then, by adding $u'_2$ and $u_2$ to $P$, we get a unique path satisfying the conditions of Rule 1, therefore $u'_2$ must have given a charge of $\frac{1}{2}$ to $u_2$ and, hence, $u_2$ does not take any charge when applying Rule 2.
\end{proof}

We are now ready to prove the Theorem~\ref{main-theorem2}.

\begin{proof}(Of Theorem~\ref{main-theorem2}.) By discharging technique, the initial charge assigned to each vertex $v$: $$c(v)=d(v)-\frac{14}{5}.$$ 
Since we have assumed that the average degree of $G$ is strictly less than $\frac{14}{5}$, the total charge is a negative value.
However, after applying the discharging rule introduced next, we will partition the vertex set so that on each part the sum of final charges is positive. This would be in contradiction with the total charge being negative and would complete the proof of the theorem.
The discharging rule is as follow:
\medskip

{\bf Discharging rule:} A $4^+$-vertex gives a charge of $\frac{2}{5}$ to each of its $2$-neighbors and a charge of $\frac{1}{5}$ to each of its $3_1$-neighbors.

\medskip

Let $c^*(v)$ be the final charge of the vertex $v$ after discharging. It is immediate that if $d(v)\geq 5$, then $c^*(v)\geq \frac{1}{5}$.
For a $4$-vertex $v$, it follows from Lemma~\ref{lem:weakvertices} and Lemma~\ref{lem:RC4_3-3_1} that $c^*(v)\geq 0$. To complete the proof, we show that the total charges on each connected component of the $3^-$-subgraph of $G$ is non-negative. 

Let $H$ be one such a component. If $H$ has no vertex which is a $2$-vertex in $G$, then all vertices have positive charges. Let $v$ be a $2$-vertex of $G$ in $H$. Observe that if $H$ consists of only $v$, then both its neighbors are $4^+$-vertices and $c^*(v)=0$. Otherwise, either $c^*(v)=-\frac{2}{5}$ or $c^*(v)=-\frac{4}{5}$. For the former to be the case, one of the neighbors of $v$ must be a $4^+$-vertex of $G$, thus $v$ has a unique neighbor in $H$. For the latter to be the case, both neighbors of $v$ must be $3$-vertices and thus $v$ has two neighbors in $H$. Let $l$ be the number of $2$-vertices of $G$ in $H$ that their final charge is $-\frac{2}{5}$ and let $k$ be the number of $2$-vertices of $G$ in $H$ that their final charge is $-\frac{4}{5}$. By Lemma~\ref{lem:weakvertices}, the neighbors of these $2$-vertices are $l+2k$ distinct $3_1$-vertices of $G$ in $H$. Of these $l+2k$ vertices in $H$, suppose $p$ of them are of degree 3 in $H$, and that the rest are either of degree $2$ or $1$, the latter being possible only when $H$ is just an edge. Observe that $3$-vertices of $G$ with at most two neighbors in $H$ have a third neighbor that is a necessarily $4^+$-vertex of $G$, and, therefore, such vertices have charge at least $\frac{2}{5}$. For the $p$ vertices that are $3_1$-vertices of $G$ in $H$, by Lemma~\ref{lem:n_0-n_1}, there must be at least $p$ other vertices in $H$ that are $3_0$-vertices of $G$. As all these vertices have a charge of $\frac{1}{5}$, the over-all charge in a connected component $H$ of the $3^-$-subgraph of $G$ is non-negative, proving our claim.
\end{proof}

\section{Mapping signed graphs to $(K_{2k}, M)$}\label{sec:K2k}

When the target signed graph is $(K_8, M)$, rather than $(K_6,M)$, we show the condition of a maximum average degree strictly smaller than $3$ is enough for $(G, \sigma)$ to map to it. The proof is quite easy in this case. In the next section, we will see that this condition is not only the best possible for $(K_8,M)$ but that it cannot be improved for $(K_{2k},M)$, $k\geq 4$, either. As in the previous case to prove our claim we will work with ${\rm DSG}(K_8, M)$.

\begin{theorem}\label{main-theorem22}
Every signed graph of maximum average degree less than $3$ admits an edge-sign preserving homomorphism to ${\rm DSG}(K_8, M)$. Moreover, the bound $3$ is best possible.
\end{theorem}

To prove Theorem~\ref{main-theorem22}, we assume that $(G, \sigma)$ is a minimum counterexample which does not admit an edge-sign preserving homomorphism to ${\rm DSG}(K_8, M)$. First we study the properties of a list homomorphism of a signed rooted tree to ${\rm DSG}(K_8, M)$.

\begin{lemma}
	\label{8oneEdgeRestriction}
	Let $xy$ be a signed edge and let $L$ be its ${\rm DSG}(K_8, M)$-list assignment. Then the following statements hold:
	\begin{enumerate}[leftmargin =3em, label=(\arabic*)]
		\item \label{8-P2-0.5} If $|L^a(x)|= 1$, then $F_x(y)$ is a paired set of size $9$.
		\item \label{8-P2-2.5} If $L^a(x)$ contains either a neighbored $5$-set or a one-sided $6$-set, then $F_x(y)$ is a paired set of size at most $2$.
	\end{enumerate}
\end{lemma}

\begin{corollary}
	\label{8-path3twoolored}
	Let $(P_3, \sigma)$ be a signed path $xvy$ and let $L$ be a ${\rm DSG}(K_8, M)$-list assignment of $(P_3, \sigma)$ with $L(v) =C$, $L(x)=\{c_x\}$ and $L(y)= \{c_y\}$. Then $C \setminus (F_x(v) \cup F_y(v))$ contains two colors which are in different layers unless one of the followings holds:
	\begin{enumerate}[leftmargin =3em, label=(\arabic*)]
		\item $c_x$ and $c_y$ are in the same layer but on different sides, and $P_3$ is a positive path;
		\item $c_x$ and $c_y$ are in the same layer and on the same side, and $P_3$ is a negative path.
	\end{enumerate}
\end{corollary}

We note that in the two special cases  $(P_3, \sigma)$ admits no $L$-coloring.

Next we list a set of forbidden configurations in the minimum counterexample $(G, \sigma)$.

\begin{lemma}\label{8onevertex}
	The signed graph $(G,\sigma)$ does not contain the following vertices: $1$-vertex, $2_1$-vertex, $3_1$-vertex, $4_3$-vertex and $5_5$-vertex.
\end{lemma}
\begin{proof}
		
	We only prove the case of a $3_1$-vertex, the remaining cases being almost direct corollary of the Lemma~\ref{8oneEdgeRestriction}. Suppose to the contrary that $v$ is a $3_1$-vertex in $G$. Let $u$ be the $2$-neighbor of $v$, let $v_1,v_2$ be the other two neighbors of $v$, and let $u_1$ be the second neighbor of $u$, see Figure~\ref{8FC3_1-3}.
	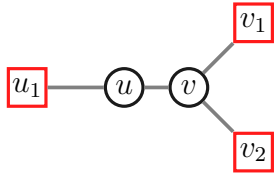
\begin{figure}[h]
		\centering
		\begin{tikzpicture}[>=latex,
		roundnode/.style={circle, draw=black!90, very thick, minimum size=5mm, inner sep=0pt},
		roundnode2/.style={circle, draw=black!90, very thick, minimum size=5mm, inner sep=0pt},
		squarednode/.style={rectangle, draw=red!90, very thick, minimum size=5mm, inner sep=0pt}, scale=0.86
		] 
		\node [roundnode] (u) at (2,1){$u$};
		\node [roundnode] (v) at (3,1){$v$};
		\node [squarednode] (u1) at (0.5,1){$u_1$};
		\node [squarednode](v3) at (4,0){$v_2$};
		\node [squarednode](v1) at (4,2){$v_1$};
		\draw [line width =1.4pt, gray] (u)--(u1);
		\draw [line width =1.4pt, gray, -] (u)--(v);
		\draw [line width =1.4pt, gray, -] (v)--(v1);
		\draw [line width =1.4pt, gray, -] (v)--(v3);
		\draw [line width =1.4pt, gray] (v1);
		\end{tikzpicture}
		\caption{A $3_1$-vertex with its neighbors.}
		\label{8FC3_1-3}
	\end{figure}
	Let $G'$ be the graph obtained from $G$ by removing $u$ and let $H$ be the subgraph induced by $\{u,v\}$. By the minimality of $(G, \sigma)$, there is an edge-sign preserving homomorphism $\phi$ of $(G',\sigma)$ to ${\rm DSG}(K_8, M)$. Since $\phi(v)$ exists, the two exceptions of Corollary~\ref{8-path3twoolored} cannot be the case here and, therefore, $C\setminus (F_{v_1}(v) \cup F_{v_2}(v))$ contains two colors which are in different layers. Let $\phi'$ be the restriction of $\phi$ on $(G, \sigma)-(H, \sigma)$ and let $L$ be an associated list assignment of $(H, \sigma)$. Now we shall show that $(H, \sigma)$ is $L$-colorable, where $L(u)=C\setminus F_{u_1}(u)$ is a neighbored $7$-set, $L(v) =C\setminus (F_{v_1}(v) \cup F_{v_2}(v))$. By Lemma \ref{8oneEdgeRestriction} \ref{8-P2-2.5}, $F_{u}(v)$ is a paired $2$-set, thus $L^a(v) = L(v)\setminus F_{u}(v) \neq \emptyset$, a contradiction.
\end{proof}

\begin{proof}(Of Theorem \ref{main-theorem22}.) By discharging method, we assign an initial charge of $$c(v)=d(v)-3$$ at each vertex $v$ of $(G, \sigma)$. Then by the assumption on the average degree we have $\sum_{v \in V(G)}c(v)<0$. We apply only the following discharging rule:

\medskip
{\bf Discharging rule} Every $2$-vertex receives $\frac{1}{2}$ from each of its two neighbors.
\medskip

It is straightforward to check that all vertices have non-negative charges after applying this rule, a contradiction with the fact that the total charge is a negative value.
\end{proof} 

\section{Tightness and planarity}\label{sec:examples}

In this section, first we will give several examples to show the tightness of our theorems. Noting that our examples are, in particular, signed planar graphs which imply that the conditions of the no-homomorphism lemma is not sufficient for mapping signed planar graph to $(K_{3,3}, M)$, $(K_6, M)$ and $(K_8, M)$ while it is sufficient for some other signed graphs such as $(K_4,\empty)$, $(K_{4,4},M)$. We then apply our results to planar graph with further structural conditions, and propose further direction of study.

\subsection{Tightness and examples}

The first example, given in Figure~\ref{fig:sharp6}, shows the bound of $\frac{14}{5}$ in Theorem~\ref{main-theorem} is sharp.
 
\begin{figure}[htpb]
	\begin{minipage}[t]{.5\textwidth}
	\centering
	\begin{tikzpicture}[>=latex,
	roundnode/.style={circle, draw=black!90, very thick, minimum size=5mm, inner sep=0pt}, scale=0.86
	] 
	\node [roundnode] (u) at (0,0){$u$};
	\node [roundnode] (v) at (4,0){$v$};
 
	\node [roundnode](x) at (2,0.7){$x$};
	\node [roundnode](y) at (2,1.7){$y$};
 
	\node [roundnode](w) at (2,3){$w$}; 
	\draw [line width =1.4pt, blue] (y)--(w)--(v)--(u)--(w);
	\draw [line width =1.4pt, blue] (v)--(x)--(u);
 	\draw [dashed, line width =1.4pt, red] (x)--(y);
	\end{tikzpicture}
	\caption{$\mad(G, \sigma)=\frac{14}{5}$}
	\label{fig:sharp6}
	\end{minipage}
		\begin{minipage}[t]{.5\textwidth}
		\centering
		\begin{tikzpicture}
		[>=latex,	
		roundnode/.style={circle, draw = black, very thick, minimum size=4mm, inner sep=0pt},
		bluenode/.style={circle, draw=black!90, very thick, minimum size=3mm, inner sep=0pt}, scale=0.86
		] 
		\centering
		\foreach \i in {1,2,3,4}
		{
			\draw[rotate=-90*(\i-2)+45] (0, 2) node[roundnode] (u_\i){};
		}
		\foreach \i in {1,2,3,4}
		{
			\draw[rotate=-90*(\i-2)+45] (0, 0.8) node[roundnode] (v_\i){};
		}
		
		\foreach \i/\j in {1/2,2/3,3/4}
		{
			\draw  [line width=1.4pt, blue] (u_\i) -- (u_\j);
		}
		\foreach \i/\j in {1/2,2/3,4/1}
		{
			\draw  [line width=1.4pt, blue] (v_\i) -- (v_\j);
		}
		\foreach \i in {1,3,4}
		{
			\draw  [line width=1.4pt, blue] (u_\i) -- (v_\i);
		}
		\foreach \i/\j in {4/1}
		{
			\draw  [dashed, line width=1.4pt, red] (u_\i) -- (u_\j);
		}
		\foreach \i/\j in {3/4}
		{
			\draw  [dashed, line width=1.4pt, red] (v_\i) -- (v_\j);
		}
		\foreach \i in {2}
		{
			\draw  [dashed, line width=1.4pt, red] (v_\i) -- (u_\i);
		}
		
		\end{tikzpicture}
		\caption{A signed graph does not map to $(K_{3,3}, M)$}
		\label{fig:bipartiteK33}
		\end{minipage}
\end{figure}

\begin{proposition}
	There exists a signed graph with maximum average degree $\frac{14}{5}$ which does not admit a homomorphism to $(K_6, M)$.
\end{proposition}

\begin{proof}
The signed graph of Figure~\ref{fig:sharp6} is an example of a signed graph of maximum average degree $\frac{14}{5}$ which we prove to not admit any homomorphism to $(K_6, M)$.
Suppose to the contrary that $(G, \sigma)$ admits a homomorphism to $(K_6, M)$ (illustrated in Figure~\ref{fig:K6}). By Theorem~\ref{thm:Sign-preserving}, there exists a switching-equivalent signature $\sigma'$ such that $(G, \sigma')$ admits an edge-sign preserving homomorphism to $(H, \pi)$. 
Observe that a positive triangle with two negative edges does not admit an edge-sign preserving homomorphism to $(K_6, M)$. Thus considering the triangles $uvw$ and $uvx$, all their edges must be positive in $\sigma'$. Hence, only one of $xy$ or $yw$ is negative. Considering the symmetry of $xy$ and $yw$, we may assume $\sigma'$ is the signature given in the figure. By symmetries of $(K_6,M)$, we may assume $xy$ is mapped to $12$. Then none of the other three vertices can map to $1$ or $2$. But then there is no positive triangle induced by $\{3, 4, 5, 6\}$ to map them to.
\end{proof}

With regard to mapping signed bipartite planar graphs to $(K_{3,3}, M)$ and $(K_{4,4},M)$, the existence of a (simple) signed bipartite planar graph all whose mapping to $(K_{4,4},M)$ are onto, is followed from a general construction of \cite{NSS16}. However, in this special case, we have a smaller example of Figure~\ref{fig:bipartiteK33}. In this example, any pair of vertices in the same part of the bipartition belongs to a negative $4$-cycle, and thus identifying any such pair would create a negative $2$-cycle. Hence, any mapping of this signed bipartite graph to $(K_{4,4}, M)$ is onto. Thus it does not map to its subgraph $(K_{3,3}, M)$.

%
%

Finally, noting that Theorem~\ref{main-theorem22} implies any graph of maximum average degree less than 3 maps to $(K_{2k}, M)$ for $k\geq 4$, we show that the conditions of maximum average degree cannot be improved for any value of $k$. Our examples, depicted in Figure~\ref{fig:couterK8}, are built from a negative cycle on each of whose edge we build a positive triangle.

\begin{figure}[h]
	\centering
	\centering
	\begin{tikzpicture}[>=latex,	
	roundnode/.style={circle, draw = black, very thick, minimum size=4mm, inner sep=0pt},
	bluenode/.style={circle, draw=black!90, very thick, minimum size=3mm, inner sep=0pt}, scale=0.86
	] 
	\node [roundnode] (A) at (0:2cm){};
		\node[bluenode] (AB) at (18:2.67cm){};
	
	\node [roundnode] (B) at (36:2cm){};
		\node[bluenode] (BC) at (54:2.67cm){};
	\node [roundnode] (C) at (72:2cm){};
		\node[bluenode] (CD) at (90:2.67cm){};
	\node [roundnode] (D) at (108:2cm){};
		\node[bluenode] (DE) at (126:2.67cm){};
	\node [roundnode] (E) at (144:2cm){};
		\node[bluenode] (EF) at (162:2.67cm){};
	\node [roundnode] (F) at (180:2cm){};
		\node[bluenode] (FG) at (198:2.67cm){};
	\node [roundnode] (G) at (216:2cm){};
		\node[bluenode] (GH) at (236:2.67cm){};
	\node [roundnode] (H) at (252:2cm){};
	\node [roundnode] (I) at (288:2cm){};
		\node[bluenode] (IJ) at (306:2.67cm){};
	\node [roundnode] (J) at (324:2cm){};
		\node[bluenode] (JA) at (342:2.67cm){};
	
	\foreach \t in {18,54,90,126,162,198,234,306,342}{	
		\node at (\t:2.2cm){$+$};}
	\node at (0:0cm){$-$};
			
	\draw[line width = 1.4pt, gray] (A)--(B)--(C)--(D)--(E)--(F)--(G)--(H);
	
	\draw[line width = 1.4pt, gray] (A)--(AB)--(B)--(BC)--(C)--(CD)--(D)--(DE)--(E)--(EF)--(F)--(FG)--(G)--(GH)--(H);
	
	\draw[line width = 1.4pt, gray] (I)--(IJ)--(J)--(JA)--(A);

	\draw[line width = 1.4pt, dotted, gray] (H)--(I);
	\draw[line width = 1.4pt, gray]	(I)--(J)--(A);
	\end{tikzpicture}
	\caption{A tight example $(G, \sigma)$}
	\label{fig:couterK8}
\end{figure}
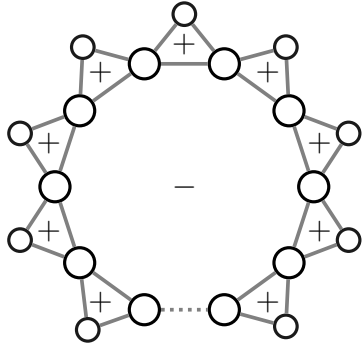

\begin{proposition}\label{thm:K8tight}
The signed graph $(G_{l}, \sigma)$, built from a negative $l$-cycle by adding a positive triangle on each edge, does not map to $(K_{2k}, M)$. 
\end{proposition}

\begin{proof}
The proof is based on the fact that $(K_{2k}, M)$ (for any given $k$) has no triangle with two negative edges. Suppose to the contrary that there exists a homomorphism of $(G_{l}, \sigma)$ to $(K_{2k}, M)$. Then, by Theorem~\ref{thm:Sign-preserving}, there exists a switching-equivalent signature $\sigma'$ and an edge-sign preserving homomorphism of $(G_{l}, \sigma')$ to $(K_{2k}, M)$. But then at least one edge of the negative $l$-cycle is assigned a negative sign by $\sigma'$ and then the triangle on this edge has two negative edges.
\end{proof}

Observe that $(G_{l},\sigma)$ is a signed planar graph that satisfies the conditions of no-homomorphism lemma with respect to $(K_{2k},M)$. We note that mapping signed bipartite planar graphs to $(K_{2k},M)$ is equivalent to mapping them to $(K_{k,k},M)$, and that mapping to the latter is a strengthening of the Four-Color Theorem as stated in Theorem~\ref{thm:4CTStrengthened}. Thus we would like to raise the following question:

\begin{problem}
For which values of $g_{01},g_{10},g_{11}$, the condition of $g_{ij}(G, \sigma)\geq g_{ij}$, $ij\in \{01,10,11\}$ would imply a mapping of signed planar graph $(G, \sigma)$ to $(K_{8}, M)$?
\end{problem}

\subsection{Application to planar graphs}

Applying Euler's formula to planar graphs, one concludes that any planar graph of girth at least $7$ has an average degree strictly less than $\frac{14}{5}$, and, since girth condition is a hereditary property, the same holds for the maximum average degree. Thus we have:
 
\begin{corollary}
If $G$ is a planar graph of girth $7$, then for every signature $\sigma$, $(G,\sigma)$ admits a homomorphism to $(K_6, M)$. 
\end{corollary}

We do not know if $7$ is the best possible girth condition in this result. On the other hand, the following restatement of Gr\"otzsch's theorem suggests a different approach.

\begin{theorem}[Gr\"otzsch's theorem restated]
Given a triangle-free planar graph $G$, the signed bipartite (planar) graph $S(G)$ maps to $(K_6,M)$. 
\end{theorem} 

In this reformulation, $S(G)$ contains negative 4-cycles, but it has no $6$-cycle. Furthermore, if $G$ is assumed to be of girth $5$, then $S(G)$ will contain no $8$-cycle either. This calls for a study in the line of Steinberg's conjecture \cite{S93} who proposed that planar graphs with no cycle of length $4,5,6$ are $3$-colorable. The conjecture is recently disproved in \cite{CHKLS17}. However, some supporting results have been proved earlier, most notable one being the result of \cite{BGRS05} which shows if cycles of length $4,5,6,7$ are not subgraphs of a planar graph $G$, then $G$ is 3-colorable.

Thus it is natural to ask:

\begin{problem}\label{conj:StrengtheningGro}
What is the smallest value of $k$, $k\geq 3$, such that every signed bipartite planar graph with no $4$-cycles sharing an edge and no cycles of length $6, 8, \ldots, 2k$, admits a homomorphism to $(K_{6}, M)$ (or equivalently to $(K_{3,3}, M)$)?
\end{problem} 
	
That such a $k$ exists follows from Theorem~\ref{main-theorem}. Indeed for $k\geq 14$ such a planar graph, by Euler's formula, will have a maximum average degree strictly less than $\frac{14}{5}$. If $k=4$ works, then this would be an strengthening of Gr\"otzsch's theorem.
	
We note that if a signed bipartite planar graph has no $C_4$, then it maps to $(K_{6}, M)$. This is a recent result of Naserasr, Pham and Wang and is based on the $4$-color theorem. 

For further study on this direction we refer to a recent work of \cite{XJ20}. In this work replacing $3$-coloring problem with homomorphism (of graphs) to $C_{2k+1}$, authors consider the question of when forbidden cycles of length $1,2, \ldots, 2k$, $2k+2, \ldots, f(k)$  and planarity imply a mapping to $C_{2k+1}$. They conclude that this is only possible when $2k+1$ is a prime number.  A natural analogue question is to ask the same for negative even cycles when signed bipartite planar graphs are considered.

\begin{figure}[h]
	\centering
	\begin{tikzpicture}
	[>=latex,	
	roundnode/.style={circle, draw = black, very thick, minimum size=4mm, inner sep=0pt},
	bluenode/.style={circle, draw=black!90, very thick, minimum size=3mm, inner sep=0pt},
	] 
	\centering
		\draw(-1, 0) node[roundnode] (x){};
		\draw(1, 0) node[roundnode] (y){};
		\draw[blue, line width =1pt](x) edge[bend right] (y);
		\draw[dashed, red, line width =1pt](x) edge[bend left] (y);
		\draw [dashed, line width=1pt, red] (x) .. controls (-1.8, 0.6) and (-1.8,-0.6) .. (x);
	
	\end{tikzpicture}
	\caption{A signed multi-graph $(D, \pi)$ on two vertices}
	\label{fig:digon}
\end{figure}
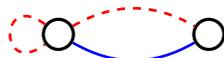

Finally we note that, if in place of $(K_6,M)$, we consider homomorphism target that are allowed to have loops or parallel edge (of different signs), then it is a restatement of the result of \cite{MRS16} that every signed planar graph of girth at least $5$ admits a homomorphism to the signed graph $(D, \pi)$ of Figure~\ref{fig:digon}.  An analogue to Steinberg question then is considered in \cite{HL18} where it is proved that every signed planar graph without cycles of length $4, 5,6,7, 8$ admits a homomorphism to $(D, \pi)$.\\

{\bf Acknowledgment.} This work is supported by the ANR (France) project HOSIGRA (ANR-17-CE40-0022). It has also received funding from the European Union's Horizon 2020 research and innovation program under the Marie Sklodowska-Curie grant agreement No 754362. It is partially supported by ARRS Projects P1-0383, J1-1692,  BI-FR/19-20-PROTEUS-001 and BI-CN/18-20-015. The Grant Number of the fourth author is NSFC 11871439, and he is also supported by Fujian Provincial Department of Science and Technology(2020J01268).

\bibliographystyle{acm}
\bibliography{homoK6}

\end{document}